\documentclass[a4paper,11]{amsart}
\usepackage[margin=2.2cm]{geometry}
\usepackage{mathrsfs}
\usepackage{mathtools}
\usepackage{amssymb}
\usepackage[all]{xy}
\usepackage{cite}
\usepackage[usenames,dvipsnames]{xcolor}
\usepackage{enumerate}
\usepackage{tikz}
\usepackage{hyperref}
\usepackage{comment}
\hypersetup{
 colorlinks,
 linkcolor={red!50!black},
 citecolor={blue!50!black},
 urlcolor={blue!80!black}
}

\AtBeginDocument{%
\def\MR#1{}
}
\theoremstyle{plain}
\newtheorem{theorem}{Theorem}[section]
\newtheorem{lemma}[theorem]{Lemma}
\newtheorem{lem}[theorem]{Lemma}
\newtheorem{prop}[theorem]{Proposition}
\newtheorem{cor}[theorem]{Corollary}
\newtheorem*{conj}{Conjecture}
\newtheorem{thm}[theorem]{Theorem}
\theoremstyle{definition}
\newtheorem{remark}[theorem]{Remark}

\newtheorem{example}[theorem]{Example}
\newtheorem{mthm}{Theorem}

\newtheorem*{mcor}{Corollary}
\numberwithin{equation}{section}
\renewcommand{\setminus}{\smallsetminus}

\newcommand{\into}{\hookrightarrow} %monomorfismo
\newcommand{\onto}{\twoheadrightarrow} %epimorfismo
\DeclareMathOperator{\rk}{rk}
\newcommand{\bmu}{\boldsymbol{\mu}} %bold mu

\DeclareMathOperator{\Spec}{Spec} %spettro di un anello
\DeclareMathOperator{\Proj}{Proj} %spazio proiettivo
\DeclareMathOperator{\Pic}{Pic}
\DeclareMathOperator{\Cl}{Cl}

\newcommand{\Newt}[1]{\mathrm{Newt}(#1)} %newton polytope
\newcommand{\Vfour}{\mathscr{V}_4} %family of smooth Fano 3-folds with anticanonical volume 4

\newcommand{\Kss}[2]{ \mathcal{M}^\mathrm{Kss}_{#1, #2} } %K-moduli stack
\newcommand{\Kps}[2]{M^\mathrm{Kps}_{#1, #2}} %K-moduli space

\newcommand\cD{\mathcal{D}}

\newcommand\cF{\mathcal{F}}
\newcommand\cG{\mathcal{G}}

\newcommand\cL{\mathcal{L}}
\newcommand\cM{\mathcal{M}}
\newcommand\cO{\mathcal{O}}
\newcommand\cR{\mathcal{R}}
\newcommand\cS{\mathcal{S}}
\newcommand\cX{\mathscr{X}}

\renewcommand\AA{\mathbb{A}}
\newcommand\CC{\mathbb{C}}

\newcommand\NN{\mathbb{N}}
\newcommand\PP{\mathbb{P}}
\newcommand\QQ{\mathbb{Q}}
\newcommand\RR{\mathbb{R}}

\newcommand\ZZ{\mathbb{Z}}

\newcommand{\Kstable}{K\nobreakdash-stable}
\newcommand{\Kstability}{K\nobreakdash-stability}
\newcommand{\Kpolystable}{K\nobreakdash-polystable}
\newcommand{\Kmoduli}{K\nobreakdash-moduli}
\newcommand{\threefold}{$3$\nobreakdash-fold}
\newcommand{\threefolds}{$3$\nobreakdash-folds}

\renewcommand{\leqslant}{\leq}
\renewcommand{\geqslant}{\geq}
\newcommand{\bP}{\mathbb{P}}
\newcommand{\bC}{\mathbb{C}}

\newcommand{\bA}{\mathbb{A}}
\newcommand{\bQ}{\mathbb{Q}}

\newcommand{\bG}{\mathbb{G}}

\newcommand{\fM}{\mathfrak{M}}
\newcommand{\vol}{\mathrm{vol}}
\newcommand{\Aut}{\mathrm{Aut}}
\newcommand{\Fut}{\mathrm{Fut}}
\newcommand{\ord}{\mathrm{ord}}
\newcommand{\cK}{\mathcal{K}}
\newcommand{\cZ}{\mathcal{Z}}
\newcommand{\cY}{\mathcal{Y}}
\newcommand{\cN}{\mathcal{N}}
\newcommand{\sV}{\mathscr{V}}
\newcommand{\sW}{\mathscr{W}}
\newcommand{\tcZ}{\widetilde{\mathcal{Z}}}
\newcommand{\tcS}{\widetilde{\mathcal{S}}}
\newcommand{\tcG}{\widetilde{\mathcal{G}}}
\newcommand{\tcY}{\widetilde{\mathcal{Y}}}
\newcommand{\tcD}{\widetilde{\mathcal{D}}}
\newcommand{\Supp}{\mathrm{Supp}}
\newcommand{\hvol}{\widehat{\mathrm{vol}}}

\begin{document}
%-------------------------------------------------------------------------------
\author[H.\ Abban]{Hamid Abban}
\address{School of Mathematical Sciences\\University of Nottingham\\Nottingham\\NG7 2RD\\United Kingdom}
\email{hamid.abban@nottingham.ac.uk}% Abban
\author[I.\ Cheltsov]{Ivan Cheltsov}
\address{School of Mathematics\\University of Edinburgh\\Edinburgh\\EH9 3FD\\United Kingdom}
\email{i.cheltsov@ed.ac.uk}% Cheltsov
\author[A.\ M.\ Kasprzyk]{Alexander Kasprzyk}
\address{School of Mathematical Sciences\\University of Nottingham\\Nottingham\\NG7 2RD\\United Kingdom}
\email{a.m.kasprzyk@nottingham.ac.uk}% Kasprzyk
\author[Y.\ Liu]{Yuchen Liu}
\address{Department of Mathematics\\Northwestern University\\Evanston\\IL 60208\\USA}
\email{yuchenl@northwestern.edu}
\author[A.\ Petracci]{Andrea Petracci}
\address{Dipartimento di Matematica\\Universit\`a di Bologna\\Piazza di Porta San Donato 5\\Bologna 40126\\Italy}
\email{a.petracci@unibo.it}% Petracci
%-------------------------------------------------------------------------------
%-------------------------------------------------------------------------------
\title{On K-moduli of quartic threefolds}
%-------------------------------------------------------------------------------
\begin{abstract}
The family of smooth Fano \threefolds{} with Picard rank~$1$ and anticanonical volume~$4$ consists of quartic \threefolds{} and of double covers of the $3$\nobreakdash-dimensional quadric branched along an octic surface. They can all be parametrised as complete intersections of a quadric and a quartic in the weighted projective space~$\PP(1,1,1,1,1,2)$, denoted by $X_{2,4}\!\subset\!\PP(1^5,2)$; all such smooth complete intersections are \Kstable. With the aim of investigating the compactification of the moduli space of quartic \threefolds{} given by \Kstability, we exhibit three phenomena: (i)\,there exist \Kpolystable{} complete intersection $X_{2,2,4}\!\subset\!\PP(1^5,2^2)$  Fano \threefolds{} which deform to quartic \threefolds{} and are neither quartic \threefolds{} nor double covers of quadric \threefolds{} -- in other words, the closure of the locus parametrising complete intersections $X_{2,4\!}\subset\!\PP(1^5,2)$ in the K-moduli contains elements that are not of this type; (ii)\,any quasi-smooth $X_{2,2,4}\!\subset\!\PP(1^5,2^2)$  is K-polystable;
(iii)\, the closure in the K-moduli space of the locus parametrising complete intersections $X_{2,2,4} \subset \PP(1^5,2^2)$ which are not complete intersections $X_{2,4} \subset \PP(1^5,2)$ contains only points which correspond to complete intersections $X_{2,2,4} \subset \PP(1^5,2^2)$.
\end{abstract}
%-------------------------------------------------------------------------------
\maketitle
%-------------------------------------------------------------------------------
\section{Introduction}
%-------------------------------------------------------------------------------
A~\emph{Fano variety} is a normal projective variety over~$\CC$ for which the anticanonical divisor is $\QQ$\nobreakdash-Cartier and ample. There has been spectacular recent progress on constructing moduli spaces of Fano varieties using \Kstability~\cite{ABHLX, xu_minimizing, BLX, jiang_boundedness, blum_xu_uniqueness, xu_zhuang, properness_K_moduli, projectivity_K_moduli_final, lwx}. For each positive integer~$n$ and for every positive rational number~$v$ there exists an Artin stack~$\Kss{n}{v}$, called the~\emph{\Kmoduli{} stack}, which is of finite type over~$\CC$ and which parametrises K\nobreakdash-semistable $n$\nobreakdash-dimensional Fano varieties with anticanonical volume~$v$. Moreover, this stack admits a good moduli space~$\Kps{n}{v}$, called the~\emph{\Kmoduli{} space}, which is projective over~$\CC$. The closed points of~$\Kps{n}{v}$ are in one-to-one correspondence with \Kpolystable{} $n$\nobreakdash-dimensional Fano varieties with anticanonical volume~$v$. We refer to~\cite{xu_survey} for a survey on these topics.

The next natural step is to investigate the geometry of K\nobreakdash-moduli spaces, providing explicit descriptions when possible. \Kmoduli{} of smooth(able) $2$\nobreakdash-dimensional Fano varieties have been studied by Mabuchi--Mukai~\cite{mabuchi_mukai} and by Odaka--Spotti--Sun~\cite{odaka_spotti_sun}. More generally, one of the most well-known classes of Fano varieties are hypersurfaces of degree~$d$ in $n$\nobreakdash-dimensional projective space~$\PP^n$, with~$d\leq n$; all smooth Fano hypersurfaces are conjectured to be \Kpolystable~\cite[Part~3]{xu_survey}. Substantial progress has made towards clarifying this conjecture~\cite[Theorem B]{AZ2}, however very little is currently known about their \Kmoduli. One readily available compact moduli space for Fano hypersurfaces is the GIT~moduli space, and Liu--Xu~\cite{liu_xu_cubic} showed that, for cubic \threefolds, the \Kmoduli{} space coincides with the GIT~moduli space. An analogous result for cubic $4$\nobreakdash-folds was recently proved by Liu~\cite{liu_cubic4folds}. There are additional results in $2$~dimensions, and in other settings~\cite{wall_crossing_K_moduli, K-moduli_quadric_K3, adl_quartic_K3, LiuPetracci}. A particular observation is the following conjecture, which is verified up to dimension $4$.
\begin{conj} K-polystable limits of cubic hypersurfaces are cubic hypersurfaces.
\end{conj}

The simplest case in higher dimensions where the expected agreement between GIT~moduli and \Kmoduli{} fails, or equivalently the conjecture above is invalid for, is the moduli space of quartic \threefolds. Here it is known that their \Kmoduli{} space contains elements that are not in the GIT~moduli.

%-------------------------------------------------------------------------------
\subsection{\Kmoduli{} of quartic \threefolds}
\label{sec:intro_quartic_3folds}
%-------------------------------------------------------------------------------
Let~$\Vfour$ denote the family of~\emph{smooth} Fano \threefolds{} with Picard rank~$1$ and anticanonical volume~$4$. Members of~$\Vfour$ are either:
\begin{enumerate}[(a)]
\item smooth quartic hypersurfaces in~$\PP^4$, simply called~\emph{smooth quartic \threefolds}; or
\item double covers of the smooth quadric \threefold{} with branch divisor of degree~$8$, often called~\emph{hyperelliptic}.
\end{enumerate}
By~\cite{cheltsov_lct_hypersurface, fujita_not_small_alpha,dervan_covers} each smooth Fano \threefold{} in the family~$\Vfour$ is \Kstable. It is not known how to characterise singular quartic \threefolds{} which are K-(poly/semi)stable. Since not every member of~$\Vfour$ is a quartic \threefold, the connected component of the \Kmoduli{} space~$\Kps{3}{4}$ containing the members of~$\Vfour$ does not coincide with the GIT~moduli space of quartic \threefolds. However, every smooth member of~$\Vfour$ is a complete intersections of type~$(2,4)$ in~$\PP(1^5,2)$; this can be seen as follows (see also~\cite[Example~4.3]{mori_complete_intersections}, the introduction of~\cite{ottem_schreieder}, and~\cite[Example~3.5.2]{calabi_problem_3folds}). Let~$x_0, \dots, x_4,y$ be the homogeneous coordinates of~$\PP(1^5,2)$ with degrees~$1,\dots,1,2$ respectively. A $(2,4)$\nobreakdash-complete intersection~$X$ is given by the vanishing of
\[
f_4(x_0, \dots, x_4) + f_2(x_0, \dots, x_4) y + a y^2 \qquad \text{and} \qquad g_2(x_0, \dots, x_4) + b y,
\]
where~$f_2$ and~$g_2$ are quadrics,~$f_4$ is a quartic, and~$a, b \in \CC$. If~$X$ is smooth then~$ a \neq 0$ or~$b \neq 0$. If~$b \neq 0$, from the second equation we can express~$y$ in terms of the~$x_i$, therefore~$y$ is a redundant variable. This implies that~$X$ is a quartic hypersurface in~$\PP^4$. If~$b = 0$ then~$a \neq 0$ so, up to scaling,~$a=1$. By completing the square (i.e.\ applying the automorphism of~$\PP(1^5,2)$ given by~$y \mapsto y + \frac{1}{2} f_2(x)$) we can assume that~$f_2 = 0$; in this case~$X$ is the double cover of the quadric \threefold{}~$\{ g_2(x) = 0 \} \subset \PP^4$ branched along the surface~$\{ f_4(x) = g_2(x) = 0 \}$.

% is expected that all \Kpolystable{} degenerations of quartic \threefolds{} are embedded in the $5$\nobreakdash-dimensional weighted projective space~$\PP(1,1,1,1,1,2)\eqqcolon\PP(1^5,2)$ as complete intersections of a quartic and a quadric. We prove below that this expectation is false: we show there are many \Kpolystable{} degenerations of quartic \threefolds{} which do not embed in~$\PP(1^5,2)$.

It is natural to ask whether the connected component of the \Kmoduli{} space~$\Kps{3}{4}$ containing the members of~$\Vfour$ parametrises $(2,4)$\nobreakdash-complete intersections in~$\PP(1^5,2)$. If this were true, it might suggest that this component of \Kmoduli{} can be constructed via some form of non-reductive~GIT by taking the moduli space of such complete intersections. We first prove that this is not the case.

\begin{mthm}[{cf.\ Proposition~\ref{prop:specific}, Theorem~\ref{theorem:IHES}, and Theorem~\ref{ALL}}] \label{thm:main}
There exist \Kpolystable{} non-smooth Fano \threefolds{} that are smoothable to a quartic \threefold{} and that are not $(2,4)$\nobreakdash-complete intersections in~$\PP(1^5,2)$.
\end{mthm}

\noindent An immediate consequence is:

\begin{mcor}
Let~$M$ be the connected component of the \Kmoduli{} space~$\Kps{3}{4}$ containing the members of~$\Vfour$. Then in $M$ there are points corresponding to \Kpolystable{} Fano \threefolds{} that are not $(2,4)$\nobreakdash-complete intersections in~$\PP(1^5,2)$.
\end{mcor}

We provide two types of examples of Fano \threefolds{} which satisfy Theorem~\ref{thm:main}: three toric examples that we study in~\S\ref{sec:toric-examples}, and an infinite family that we study in~\S\ref{sec:infinite-family}. These examples are all $(2,2,4)$\nobreakdash-complete intersections in the weighted projective $6$-space $\PP(1^5,2^2)$.

We have seen that some K-polystable degenerations of smooth quartic \threefolds{} are hyperelliptics; in other words, K-moduli limits of quartics include hyperelliptics.
In Theorem~\ref{thm:main} we have seen that K-moduli limits of hyperelliptics include Fano \threefolds which are $(2,2,4)$-complete intersections in $\PP(1^5,2^2)$ which are not $(2,4)$-complete intersections in $\PP(1^5,2)$.
We call such \threefolds{} \emph{pure} $(2,2,4)$-complete intersections.
It is natural to wonder whether pure $(2,2,4)$-complete intersections degenerate further in the K-moduli space.
Surprisingly, the K-moduli limits of pure $(2,2,4)$-complete intersections are pure $(2,2,4)$-complete intersections:

%Reiterating this phenomena, K-moduli limits of quartics contains hyperelliptics and their limits contains elements purely of type . Surprisingly, the limit of ``pure'' $(2,2,4)$-intersections only contains elements of this type. By pure we mean those elements that are not hyperelliptic or quartic. This is investigated in \S\ref{limits} by showing that the K-moduli of \threefolds{} of pure type $(2,2,4)$-complete intersection in $\PP(1^5,2^2)$ is determined by the K-moduli of pairs $(S,\frac{1}{16}\Delta)$, where $S$ is a degree $4$ del Pezzo surface and $\Delta\sim -4K_S$. This leads to:

\begin{mthm}[{cf.\ Proposition~\ref{prop:dP4} and Corollary~\ref{cor:moduli}}]
	In the K-moduli space $\Kps{3}{4}$, the locus of pure $(2,2,4)$- complete intersections in~$\PP(1^5,2^2)$ is closed.
\end{mthm}

This is proved in \S\ref{limits} by showing that the K-moduli locus of  pure  $(2,2,4)$-complete intersections in $\PP(1^5,2^2)$ is determined by the K-moduli space of pairs $(S,\frac{1}{16}\Delta)$, where $S$ is a degree $4$ del Pezzo surface and $\Delta\sim -4K_S$.
In \S\ref{wall} we provide a complete study of the wall-crossing for the pairs $(S,c\Delta)$, as $c$ is a rational number in $(0, \frac{1}{16}]$.

%-------------------------------------------------------------------------------
\subsection{Degenerations of quartic \threefolds{} inside~$\PP(1^5,2^2)$}
%-------------------------------------------------------------------------------
Denote the variables of the weighted projective space~$\PP(1^5,2^2)$ by~$x_0,\ldots,x_4,y_0,y_1$, where~$y_0$ and~$y_1$ are the two variables of weight~$2$. Let $X\subset \PP(1^5,2^2)$ be a complete intersection of two quadrics and a quartic. If~$X$ is not a cone then, after a suitable change of coordinates, it is defined by the equations
\begin{align*}
&y_0y_1=f(x_0,x_1,x_2,x_3,x_4)\\
&g(x_0,x_1,x_2,x_3,x_4)=ay_0+by_1\\
&h(x_0,x_1,x_2,x_3,x_4)=cy_0+dy_1
\end{align*}
where~$f$ has degree~$4$,~$g$ and~$h$ have degree~$2$, and~$a,b,c,d \in \CC$. There are three possibilities depending on the rank of the matrix
\[
A=
\begin{pmatrix}
a&b\\
c&d
\end{pmatrix}.
\]
If~$\rk A=2$ then~$X$ is a quartic hypersurface in~$\PP^4$. If~$\rk A=1$ then~$X$ is hyperelliptic; i.e.\ a $(2,4)$\nobreakdash-complete intersection in~$\PP(1^5,2)$. If~$A$ is the null matrix then~$X$ has two singular points at $p_0=[0:0:0:0:0:1:0]$ and~$p_1= [0:0:0:0:0:0:1]$.  Let us now assume that we are in the latter case, so that~$X \subset \PP(1^5,2^2)$ is defined by
\begin{align*}
&y_0 y_1 =f(x_0,x_1,x_2,x_3,x_4)\\
&g(x_0,x_1,x_2,x_3,x_4)=0\\
&h(x_0,x_1,x_2,x_3,x_4)=0
\end{align*}
where~$f$ is a quartic and~$g$ and~$h$ are quadrics. In~\S\ref{sec:toric-examples} we consider special binomials for~$f$,~$g$,~$h$ and we study certain toric varieties~$X$.
In~\S\ref{sec:infinite-family} we study the infinite family of varieties~$X$ obtained by picking general~$f$,~$g$,~$h$. In~\S\ref{limits} we first show that if the surface $\{g\!=\!h\!=\!0\}\subset\PP^4$ and the curve $\{f\!=\!g\!=\!h\!=\!0\}\subset\PP^4$ are smooth then $X$ is K-polystable. We then prove that K-polystable degenerations of these complete intersection $X_{2,2,4}\!\subset\!\PP(1^5,2^2)$ are also complete intersections of these type.

%-------------------------------------------------------------------------------
\subsection{Overview of the proof}
\label{methods}
%-------------------------------------------------------------------------------
There are several methods available to verify K\nobreakdash-(poly)stability of a given Fano variety. For the toric examples, we use the most natural tool: a toric variety is \Kpolystable{} if and only if the barycentre of its anticanonical polytope is the origin. This translates the algebro-geometric condition into a combinatorial one amenable to computer-assisted investigation. The toric examples which satisfy Theorem~\ref{thm:main} were found via a computer search amongst Fano polytopes using techniques from Mirror Symmetry~\cite{procams,quantum_3folds,fanosearch}. This is explained in~\S\ref{sec:fanosearch}; the examples are studied in~\S\ref{sec:toric-examples}.

For the infinite family of examples satisfying Theorem~\ref{thm:main}, we first use estimates on stability thresholds. Two of the most useful tools in \Kstability{} are the Fujita--Li valuative criterion~\cite{Li-beta,Fujita-beta} and the stability threshold, also called the~$\delta$-invariant, introduced in~\cite{fujita-odaka}. It follows that a Fano variety~$X$ is \Kstable{} if and only if~$\delta(X)>1$. Although~$\delta(X)$ is extremely difficult to compute in general, a method to find lower bounds for~$\delta(X)$ is described in~\cite{AbbanZhuang}. Roughly speaking, one chooses a flag over~$X$ and computes the refinement of the anticanonical linear system with respect to this flag~\cite[\S2]{AbbanZhuang}; the refinement provides a lower bound for~$\delta(X)$. With a little care, one can choose the flag so that the stability threshold of the refinement is greater than~$1$, implying \Kstability. We do this for the infinite family in~\S\ref{sec:infinite-family}. In~\S\ref{limits}, we use techniques of cyclic covers, cone construction, degeneration and interpolation to obtain two results: we first prove that all complete intersections $X_{2,2,4}$ as described above are K-polystable as long as $\{g\!=\!h\!=\!0\}\subset\PP^4$ and $\{f\!=\!g\!=\!h\!=\!0\}\subset\PP^4$ are smooth; then
we prove that the closure in the K-moduli space of the locus of complete intersections $X_{2,2,4} \subset \PP(1^5,2^2)$ which are not complete intersections $X_{2,4} \subset \PP(1^5,2)$ contains only points which correspond to complete intersections $X_{2,2,4} \subset \PP(1^5,2^2)$.

%-------------------------------------------------------------------------------
\subsection{Connection to Mirror Symmetry}
\label{sec:fanosearch}
%-------------------------------------------------------------------------------
The three toric \threefolds{} presented in Proposition~\ref{prop:specific} and in Remark~\ref{rmk:two_more_examples} were found using a computer-assisted search guided by expectations arising from Mirror Symmetry for Fano varieties~\cite{procams, quantum_3folds, fanosearch}. It is expected that deformation families of smooth (or maybe mildly singular, i.e.\ with orbifold terminal singularities) Fano varieties of dimension~$n$ are in one-to-one correspondence with mutation-equivalence classes of certain `special' Laurent polynomials in~$n$ variables. We need to specify:~\eqref{item:mutation} the meaning of the word~\emph{mutation};~\eqref{item:special} the meaning of the adjective `special'; and~\eqref{item:correspondence} how the correspondence works.

\begin{enumerate}[\hspace{-5pt}(i)]
\item\label{item:mutation}
By~\emph{mutation} we mean a generalisation, introduced in~\cite{sigma}, of the notion of mutation in cluster algebra theory~\cite{cluster}. We do not give the definition here: see~\cite[Definition~2]{sigma} for details. It is sufficient to know that, under certain circumstances, a Laurent polynomial~$f \in \QQ[x_1^{\pm1}, \dots, x_n^{\pm1}]$ can be mutated to another Laurent polynomial~$g \in \QQ[x_1^{\pm1}, \dots, x_n^{\pm1}]$. Two Laurent polynomials~$f$ and~$g$ are said to be mutation-equivalent if there exists a finite sequence of mutations transforming $f$ to $g$.
\item\label{item:special}
In place of `special' we should have written~\emph{rigid maximally mutable}. The class of rigid maximally mutable Laurent polynomials (or~\emph{rigid~MMLPs} for short) was introduced in~\cite{procams} for dimension~$2$, and in~\cite{maximally_mutable} for any dimension. We remark that if~$f \in \QQ[x_1^{\pm1}, \dots, x_n^{\pm1}]$ is a rigid~MMLP in~$n$ variables, then its Newton polytope~$P\coloneqq\Newt{f} \subset \RR^n$ is an~$n$\nobreakdash-dimensional lattice polytope such that the origin lies in the interior of~$P$, and each vertex of~$P$ is a primitive lattice vector:~$P$ is an example of a~\emph{Fano polytope}. One can consider the \emph{spanning fan} (or \emph{face fan}) whose cones are generated by the faces of~$P$, and the (possibly singular) Fano toric variety~$X_P$ associated to this spanning fan.
\item\label{item:correspondence}
The correspondence between deformation families of Fano varieties and mutation-equivalence classes of rigid~MMLPs is described in~\cite{procams,maximally_mutable,CHK22}. Briefly, we expect it to work as follows: given a rigid~MMLP $f$ with Newton polytope~$P$, we consider the Fano toric variety~$X_P$ and associate a (partial) smoothing of~$X_P$. The reason why this assignment should be well-defined is the following result due to Ilten~\cite{ilten_sigma}: if~$f$ and~$g$ are Laurent polynomials related via a mutation, then the Fano toric varieties~$X_P$ and~$X_Q$ are deformation equivalent (here~$Q$ is the Newton polytope of~$g$); i.e.\ there exists a flat proper family~$\cX \to \PP^1$ such that the fibre over~$0 \in \PP^1$ is~$X_P$ and the fibre over~$\infty \in \PP^1$ is~$X_Q$. Notice that a singular Fano toric variety can have many different smoothings  (e.g.~\cite[Theorem~3.1]{kaloghiros_petracci} or~\cite{petracci_roma}); one needs to select the smoothing of~$X_P$ which is compatible with the mutations of~$f$.
\end{enumerate}

Now we explain how we found the toric examples. We start from the polytope~$P'$ whose vertices are the vectors in~\eqref{eq:vectors_symmetric_quartic} in Remark~\ref{rmk:balanced_quartic}. The toric variety~$X_{P'}$ associated to the spanning fan of~$P'$ is the singular quartic hypersurface~$\{ x_1 x_2 x_3 x_4 = x_0^4 \} \subset \PP^4$, which clearly deforms to members of~$\Vfour$. There exists a unique rigid~MMLP $f$ such that~$P' = \Newt{f}$, namely
\[
f = \frac{(1+x+y+z)^4}{xyz} -24.
\]
We used the computer algebra system~\textsc{Magma}~\cite{BCP97} to construct many Laurent polynomials~$g$ which are mutation-equivalent to~$f$. Each of these~$g$'s gives a Fano toric variety~$X_Q$, where $Q\coloneqq\Newt{g}$, which, according to the Mirror Symmetry expectations described above, should be a degeneration of quartic \threefolds. We then filter for those~$g$ such that the polar of~$Q$ has barycentre at the origin; this is equivalent to requiring that the Fano toric variety~$X_Q$ is \Kpolystable. In this way we found three Laurent polynomials, one of which is given by:
\begin{align*}
	g &= x^3 y^2 z^4 + 2 x^2 y^3 z^2 + 4 x^2 y^2 z^2 + 8 x^2 y z^3 + 2 x^2 y z^2
	+ x y^4 + 4 x y^3 + 8 x y^2 z + 6 x y^2 \\
	&+ 16 x y z + 4 x y + 28 x z^2 + 8 x z
+ x + 12 y + 56 z/y + 12/y + 8/(x z) + 16/(x y z) + 70/(x y^2) + 8/(x y^2 z) \\
	&+ 2/(x^2 y z^2) + 4/(x^2 y^2 z^2) + 56/(x^2 y^3 z) + 2/(x^2 y^3 z^2)
+	28/(x^3 y^4 z^2) + 8/(x^4 y^5 z^3) + 1/(x^5 y^6 z^4).
\end{align*}
The Newton polytope~$P\coloneqq\Newt{g}$ is the Fano polytope given in Proposition~\ref{prop:specific}. The Newton polytopes of the remaining two Laurent polynomials found using this method are described in Remark~\ref{rmk:two_more_examples}.

%-------------------------------------------------------------------------------
\subsection*{Notation and conventions}
\label{sec:notation}
%-------------------------------------------------------------------------------
We work over an algebraically closed field of characteristic zero, denoted by~$\CC$. Every toric variety or toric singularity is assumed to be normal.

%-------------------------------------------------------------------------------
\section{The toric examples}
\label{sec:toric-examples}
%-------------------------------------------------------------------------------
We begin by analysing a toric Fano  \threefold{} satisfying Theorem~\ref{thm:main}.

\begin{prop} \label{prop:specific}
Let~$P$ be the polytope with vertices
\begin{equation} \label{eq:vertices_of_P}
\begin{pmatrix}
3 \\ 2 \\ 4
\end{pmatrix},\
\begin{pmatrix}
1 \\ 4 \\ 0
\end{pmatrix},\
\begin{pmatrix}
1 \\ 0 \\ 0
\end{pmatrix},\
\begin{pmatrix}
-5 \\ -6 \\ -4
\end{pmatrix}
\end{equation}
in the lattice~$ N = \ZZ^3$ and let~$X$ be the toric variety associated to the spanning fan of~$P$.
Then:
\begin{enumerate}
\item\label{specific_1}
$X$ is a~$\QQ$-factorial \Kpolystable{} Fano \threefold;
\item\label{specific_2}
$\Pic(X) \simeq \ZZ$ and~$\Cl(X) \simeq \ZZ \oplus \ZZ / 2 \ZZ \oplus \ZZ / 8 \ZZ$;
\item\label{specific_3}
$X$ is the quotient~$\PP^3 / (\bmu_2 \times \bmu_8)$, where~$\bmu_2$ acts linearly with weights~$(0,1,0,1)$ and~$\bmu_8$ acts linearly with weights~$(0,5,1,6)$;
\item\label{specific_4}
the singular locus of~$X$ consists of six rational curves generically along which~$X$ has transverse~$A_1$,~$A_3$ or~$A_7$ singularities;
\item\label{specific_5}
there are exactly two non-Gorenstein points on~$X$ and both of them have Gorenstein index~$2$ and are not canonical;
\item\label{specific_6}
$X$ is not a $(2,4)$\nobreakdash-complete intersection in~$\PP(1^5, 2)$;
\item\label{specific_7}
$X$ is the $(2,2,4)$\nobreakdash-complete intersection in~$\PP(1^5, 2^2)$ given by the equations
\begin{align*}
x_0^4 - y_0 y_1 &= 0\\
x_2^2 - x_0 x_3 &= 0\\
x_3^2 - x_1 x_4 &= 0
\end{align*}
where~$x_0, \ldots, x_4, y_0, y_1$ are the homogeneous coordinates of~$\PP(1^5, 2^2)$ with degrees $1,\ldots,1,2,2$ respectively;
\item\label{specific_8}
$X$ deforms to a quartic \threefold.
\end{enumerate}
\end{prop}

\begin{proof}[Proof of~\eqref{specific_1}]
Let~$\Sigma$ be the spanning fan of~$P$. It is clear that~$X$ is a Fano \threefold. Since~$P$ is a tetrahedron, each cone of~$\Sigma$ is simplicial, therefore~$X$ is $\QQ$\nobreakdash-factorial. Let~$M$ be the lattice dual to~$N$ and let~$\langle \cdot , \cdot \rangle \colon M \times N \to \ZZ$ be the dual pairing. We denote by the same symbol its extension to the associated real vector spaces, i.e.\ $M_\RR \times N_\RR \to \RR$.
Consider the polar~$P^\circ$ of~$P$:
\[
P^\circ \coloneqq \{ u \in M_\RR \mid \langle u, v \rangle \geq -1\text{ for all }v\in P \}.
\]
This is the polytope associated to the toric boundary of~$X$, which is an anticanonical divisor. One can show that~$P^\circ$ is the rational polytope in~$M_\RR$ with vertices $(-1, 0, \frac{3}{2})$, $(-1, 1, 0)$, $( 3, -1, -2)$, $(-1, 0, \frac{1}{2})$. Since the barycentre of~$P^\circ$ is the origin,~$X$ is \Kpolystable{} by~\cite{berman_polystability}.
\end{proof}

\begin{proof}[Proof of~\eqref{specific_2} and of~\eqref{specific_3}]
Consider the linear map~$\rho \colon \ZZ^4 \to N = \ZZ^3$ which maps the~$i$th standard basis vector to the~$i$th vertex of~$P$ in~\eqref{eq:vertices_of_P}. Consider the transpose~$^t \! \rho \colon M = \ZZ^3 \to \ZZ^4$. This is injective and its cokernel is isomorphic to~$\ZZ \oplus \ZZ / 2\ZZ \oplus \ZZ / 8\ZZ$ via the homomorphism~$\ZZ^4 \to \ZZ \oplus \ZZ / 2\ZZ \oplus \ZZ / 8\ZZ$ given by the matrix
	\begin{equation*}
		\begin{pmatrix}
			1 & 1 & 1 & 1 \\
			\overline{0}^2 & \overline{1}^2 & \overline{0}^2 & \overline{1}^2 \\
			\overline{0}^8 & \overline{5}^8 & \overline{1}^8 & \overline{6}^8
		\end{pmatrix},
	\end{equation*}
where~$\overline{\cdot}^2$ (respectively, $\overline{\cdot}^8$) denotes the reduction modulo~$2$ (respectively, $8$).
	The short exact sequence
	\[
	0 \longrightarrow M \overset{^t \! \rho}{\longrightarrow} \ZZ^4 \longrightarrow \ZZ \oplus \ZZ / 2 \ZZ \oplus \ZZ / 8\ZZ \longrightarrow 0
	\]
	is the divisor sequence~\cite[Theorem~4.1.3]{cls}, hence the divisor class group of~$X$ is isomorphic to~$\ZZ \oplus \ZZ / 2 \ZZ \oplus \ZZ / 8\ZZ$.
	By~\cite[Proposition~4.2.5]{cls} the Picard group of~$X$ is free; moreover, since~$X$ is $\QQ$\nobreakdash-factorial, the Picard rank of~$X$ coincides with the rank of~$\Cl(X)$, which is~$1$.
	The presentation of~$X$ as a finite abelian quotient of~$\PP^3$ follows from~\cite[Exercise~5.1.13]{cls}.
\end{proof}

\begin{proof}[Proof of~\eqref{specific_4}]
	Let~$\rho_1, \rho_2, \rho_3, \rho_4$ be the elements in~$N$ which appear in~\eqref{eq:vertices_of_P}. For~$1 \leq i < j \leq 4$, let~$\sigma_{ij} \in \Sigma$ the $2$\nobreakdash-dimensional cone with rays~$\RR_{\geq 0} \rho_i$ and~$\RR_{\geq 0} \rho_j$.
	Let~$C_{ij}$ be the closure of the torus orbit on~$X$ associated to the cone~$\sigma_{ij}$: it is a smooth rational curve.
	The shape of the cone~$\sigma_{ij}$ determines the singularities of~$X$ generically along~$C_{ij}$. We write down the analysis for~$C_{12}$ and we omit the other cases.
	
	We consider the~$2$\nobreakdash-dimensional lattice~$N_{12} = N \cap (\RR \rho_1 + \RR \rho_2)$.
	This is the orthogonal of~$(8,-2,-5) \in M$.
	The vectors~$\rho_1$ and~$\rho_2$ form an~$\RR$-basis of~$\RR \rho_1 + \RR \rho_2 = N_{12} \otimes_\ZZ \RR$, but not a~$\ZZ$-basis of~$N_{12}$.
	The finite abelian group~$N_{12} / (\ZZ \rho_1 + \ZZ \rho_2)$ has order~$2$ and its generator is the class of
	\[
	\frac{1}{2} \rho_1 + \frac{1}{2} \rho_2 = \begin{pmatrix}
		2 \\ 2 \\ 2
	\end{pmatrix} \in N_{12}.
	\]
	This implies that the toric variety associated to the cone~$\sigma_{12}$ in the lattice~$N_{12}$ is the~$A_1$ surface singularity~$\Spec \CC[x,y,z] / (xy-z^2)$.
	This implies that~$X$ has transverse~$A_1$ singularities generically along~$C_{12}$.	
	\end{proof}

\begin{proof}[Proof of~\eqref{specific_5}]
	By~\eqref{specific_4} the non-Gorenstein locus of~$X$ is contained in the set of the~$4$ torus-fixed points: $p_{123},p_{124},p_{134}$, and $p_{234}$.
	Here~$p_{ijk}$ is the torus-fixed point on~$X$ corresponding to the $3$\nobreakdash-dimensional cone~$\sigma_{ijk}$ with rays~$\RR_{\geq 0} \rho_i$,~$\RR_{\geq 0} \rho_j$, and~$\RR_{\geq 0} \rho_k$.
	We need to analyse the singularities of~$X$ at these points.
	
	Let us start from~$p_{123}$.
	The three vectors~$\rho_1$,~$\rho_2$, and~$\rho_3$ lie on the affine plane
	\[
	H_{(2, 0, -1), 2} \coloneqq \{ v \in N_\RR \mid \langle (2, 0, -1), v \rangle = 2 \}.
	\]
Since the lattice vector~$(2,0,-1) \in M$ is primitive, we get that the Gorenstein index of~$X$ at~$p_{123}$ is~$2$. Moreover, the singularity~$p_{123} \in X$ is not canonical because the lattice vector
	\[
	\begin{pmatrix}
		1 \\ 1 \\ 1
	\end{pmatrix} = \frac{1}{4} \rho_1 + \frac{1}{8} \rho_2 + \frac{1}{8} \rho_3
	\]
	lies in the interior of the polytope with vertices~$0, \rho_1, \rho_2, \rho_3$ (see~\cite[Proposition~11.4.12b]{cls}).
	
	Now consider~$p_{124}$.
	The three vectors~$\rho_1$,~$\rho_2$, and~$\rho_4$ lie on the affine plane
	\[
	H_{(-3, 1, 2), 1} \coloneqq \{ v \in N_\RR \mid \langle (-3, 1, 2), v \rangle = 1 \}.
	\]
	Since the lattice vector~$(-3,1,2) \in M$ is primitive, we get that the Gorenstein index of~$X$ at~$p_{123}$ is~$1$. By~\cite[Proposition~11.4.11]{cls} the singularity~$p_{124} \in X$ is canonical.
	
	In an analogous way we can prove that~$p_{134}$ is a Gorenstein canonical singularity, whereas~$p_{234}$ is a non-canonical singularity with Gorenstein index~$2$.
	\end{proof}

\begin{proof}[Proof of~\eqref{specific_6}]
	For brevity, set~$\PP\coloneqq\PP(1^5,2)$. The singular locus of~$\PP$ consists of a single point $p = [0 : 0 : 0 : 0 : 0 : 1]$.
For a contradiction assume that~$X$ is a $(2,4)$\nobreakdash-complete intersection in~$\PP$ and consider the corresponding closed embedding~$X \into \PP$.
	We have that~$X \setminus \{ p \} \into \PP \setminus \{ p \}$ is a regular closed embedding, because it is locally defined by the vanishing of the dehomogeneisations of the quartic and the quadric that define~$X$ inside~$\PP$.
	This implies that~$X \setminus \{ p \}$ is lci, and in particular Gorenstein.
	Therefore the non-Gorenstein locus of~$X$ is contained in~$\{ p \}$, so it is either empty or consists of one point.
	This contradicts~\eqref{specific_5}.
\end{proof}

\begin{proof}[Proof of~\eqref{specific_7}]
	In the lattice~$M \oplus \ZZ$ we consider the cone~$\tau$ whose apex is at the origin and which is spanned by~$P^\circ \times \{ 1 \}$.
	In other words,~$\tau$ is the cone over~$P^\circ$ placed at height~$1$.
	The primitive generators of~$\tau$ are
	\begin{align*}
		y_1 &= (-2, 0, 3,2), \\
		x_1 &= (-1, 1, 0,1), \\
		x_4 &= (3, -1, -2,1), \\
		y_0 &= (-2, 0, 1,2).
	\end{align*}
	The Hilbert basis (i.e.\ the minimal set of generators) of the monoid~$\tau \cap (M \oplus \ZZ)$ is made up of~$y_1$, $x_1$, $x_4$, $y_0$, $x_0 = (-1,0,1,1)$, $x_2 = (0,0,0,1)$, $x_3 = (1, 0, -1, 1)$.
	
	Since~$P^\circ$ is the moment polytope of the toric boundary of~$X$, which is anticanonical and hence ample, we have that~$X = \Proj \CC[\tau \cap (M \oplus \ZZ)]$, where the~$\NN$-grading is given by the projection~$M \oplus \ZZ \onto \ZZ$.
	This shows that~$X$ is a closed subvariety of the weighted projective space~$\PP(1^5,2^2)$, equipped with homogeneous coordinates~$x_0,\ldots,x_4,y_0,y_1$.
	It is easy to see that the equations of~$X$ in~$\PP(1^5,2^2)$ are
	\begin{align*}
		x_0^4 - y_0 y_1 &= 0,
		\\ x_2^2 - x_0 x_3 &= 0,
		\\ x_3^2 - x_1 x_4 &= 0.
	\end{align*}
Thus~$X$ is a $(2,2,4)$\nobreakdash-complete intersection in~$\PP(1^5,2^2)$.
\end{proof}

\begin{proof}[Proof of~\eqref{specific_8}]
	By using an argument similar to the one that appears in the introduction, one can see that the general $(2,2,4)$\nobreakdash-complete intersection in~$\PP(1^5,2^2)$ is a quartic \threefold.
\end{proof}

\begin{remark} \label{rmk:two_more_examples}
In addition to the example presented in Proposition~\ref{prop:specific}, we found two additional Fano toric \threefolds{} which satisfy Theorem~\ref{thm:main}.
\begin{enumerate}
	\item Let $X$ be the toric \threefold{} associated to the spanning fan of the polytope in $\ZZ^3$ with vertices
	$( 1, 3, 2)$,
	$( 1, 3, 0)$,
	$( 1, 0, 2)$,
	$( 1, 0, 0)$,
	$(-1, -1, 2)$,
	$(-1, -1, -4)$,
	$(-1, -2, 2)$,
	$(-1, -2, -4)$.
	There are isomorphisms $\Cl(X) \simeq \ZZ^5 \oplus \ZZ/2\ZZ$ and $\Pic(X) \simeq \ZZ$.
	One can prove that 	$X$ is the closed subvariety of $\PP(1^5,2^2)$ defined by the equations
	\begin{align*}
		x_0^3 x_3 - y_0 y_1	 &=0, \\
		x_2^2-x_0 x_3	 &=0, \\
		x_2 x_3-x_1 x_4	 &= 0.
	\end{align*}
	
	\item Let $X$ be the toric \threefold{} associated to the spanning fan of the polytope in $\ZZ^3$ with vertices
	$( 3, 4, 4)$,
	$( 3, 2, 4)$,
	$( 1, 2, 0)$,
	$( 1, 0, 0)$,
	$(-1, 0, 0)$,
	$(-1, -2, 0)$,
	$(-3, -2, -4)$,
	$(-3, -4, -4)$.
	There are isomorphisms $\Cl(X) \simeq \ZZ^5 \oplus \ZZ/2\ZZ \oplus \ZZ / 4 \ZZ$ and $\Pic(X) \simeq \ZZ$.
		One can prove that 	$X$ is the closed subvariety of $\PP(1^5,2^2)$ defined by the equations
		\begin{align*}
		x_0^2 x_4^2 - y_0 y_1 &=0, \\
		x_2^2 - x_0 x_4 &=0, \\
		x_1 x_3 - x_0 x_4 &= 0.
	\end{align*}
\end{enumerate}
Each of these examples is a $(2,2,4)$\nobreakdash-complete intersection in $\PP(1^5,2^2)$ and has two singular points with Gorenstein index $2$.
\end{remark}

\begin{remark} \label{rmk:balanced_quartic}
In addition to the three examples presented in Proposition~\ref{prop:specific} and in Remark~\ref{rmk:two_more_examples} we know a further \Kpolystable{} Fano toric \threefold{} which deforms to quartic \threefolds: this is the toric variety associated to the spanning fan of the polytope with vertices
\begin{equation} \label{eq:vectors_symmetric_quartic}
\begin{pmatrix}
	3 \\ -1 \\ -1
\end{pmatrix}\!, \
\begin{pmatrix}
	-1 \\ 3 \\ -1
\end{pmatrix}\!, \
\begin{pmatrix}
	-1 \\ -1 \\ 3
\end{pmatrix}\!, \
\begin{pmatrix}
	-1 \\ -1 \\ -1
\end{pmatrix}
\end{equation}
in $N = \ZZ^3$, and it is the $\QQ$-factorial hypersurface $\{ x_1 x_2 x_3 x_4 - x_0^4 =0 \}$ in $\PP^4$.
We conjecture that these four toric varieties are the only \Kpolystable{} Fano toric \threefolds{} that deform to quartic \threefolds.
\end{remark}

\begin{remark} A Fano polytope $P$ is called~\emph{symmetric} if the only point which is fixed by every automorphism of $P$ is the origin.
	The examples in Remark~\ref{rmk:two_more_examples} are Fano polytopes which are not simplices and are such that their polar has barycentre at the origin.
	 \end{remark}

\begin{remark} According to~\cite{hall} it is expected that in every mutation-equivalence class of $2$\nobreakdash-dimensional Fano polytopes there is at most one polytope $P$ such that its polar $P^\circ$ has barycentre at the origin. This is not true in higher dimension; for instance, the four polytopes presented in Proposition~\ref{prop:specific}, Remark~\ref{rmk:two_more_examples}, and Remark~\ref{rmk:balanced_quartic} are mutation-equivalent and the polar of each of them has barycentre at the origin. \end{remark}

\section{The infinite family}\label{sec:infinite-family}

\subsection{Statement}

The infinite family of examples satisfying Theorem~\ref{thm:main} is given by the following:

\begin{theorem}	\label{theorem:IHES}
	Let $x_0, \ldots, x_4, y_0, y_1$ be the homogeneous coordinates of the weighted projective space $\PP(1^5, 2^2)$ with degrees $1,\ldots,1,2,2$ respectively.
	Consider $x_0, \dots, x_4$ also as the homogeneous coordinates of $\PP^4$.
	Let $f$ be a quartic in $x_0, \dots, x_4$ and let $g$ and $h$ be two quadrics in $x_0, \dots, x_4$.
	Consider the zero-loci
	\begin{equation}\label{eq:equations_of_X_diagonal}
	\begin{aligned}
		\Delta &= \{ f = 0, \ g = h = 0 \} \subset \PP^4, \\
		S &= \{ g = h=0 \} \subset \PP^4, \\
		X &= \left\{ y_0y_1 - f = 0, \ g = h = 0 \right\} \subset \PP(1^5,2^2).
	\end{aligned}
	\end{equation}
	Assume that $\Delta$ is a smooth curve and that $S$ is a smooth surface.
	Then the following statements hold.
	\begin{enumerate}
		\item\label{item:IHES_1} The surface $S$ is a smooth del Pezzo of degree $4$ and contains exactly $16$ lines; their intersection points form a finite subset $\Sigma \subset S$ consisting of $40$ points.
		\item\label{item:IHES_2} $X$ is a klt Fano \threefold{} and the group $G = \CC^* \rtimes \ZZ / 2\ZZ$ acts faithfully on $X$.
		%\item If
	%\begin{equation}
	%			\begin{aligned}
	%		f &=\varepsilon_0x_0^4+\varepsilon_1x_1^4+\varepsilon_2x_2^4+\varepsilon_3x_3^4+\varepsilon_4x_0^4, \\
	%		g&=x_0^2+x_1^2+x_2^2+x_3^2+x_4^2, \\
	%		h&=\epsilon_0x_0^2+\epsilon_1x_1^2+\epsilon_2x_2^2+\epsilon_3x_3^2+\epsilon_4x_4^2
	%		\end{aligned}
	%\end{equation}
	%where $\epsilon_0,\epsilon_1,\epsilon_2,\epsilon_3,\epsilon_4 \in \CC$ are pairwise distinct and $\varepsilon_0,\varepsilon_1,\varepsilon_2,\varepsilon_3,\varepsilon_4 \in \CC$ are sufficiently general, then
	%$X$ is \Kpolystable.
		\item\label{item:IHES_3} If $\Sigma\cap\Delta=\varnothing$, then $X$ is \Kpolystable.
	\end{enumerate}
\end{theorem}

\begin{remark}
	The last condition $\Sigma\cap\Delta=\varnothing$ is satisfied for general choices of $f$, $g$, $h$.
	Therefore a general $X$ in Theorem~\ref{theorem:IHES} is \Kpolystable.
%\end{remark}
%\begin{remark}
	We believe that the condition $\Sigma\cap\Delta=\varnothing$ can be removed from Theorem~\ref{theorem:IHES}\eqref{item:IHES_3}.
\end{remark}

\begin{remark}
	Note that if $X \subset \PP(1^5,2)$ is a hyperelliptic cone, i.e.\ the term $y^2$ does not appear in the degree $4$ equation, or if $X\subset\PP(1^5,2^2)$ is a cone, i.e.\ the quadratic part in the variables $y_0$, $y_1$ of the degree $4$ equation $f$ does not have rank $2$, then $X$ is unstable. This can be checked rather easily using \cite[Theorem~1.4]{ZhangZhou} by setting $r=1/2$, $c=0$, and $n=3$ to obtain $\delta(X)\leq 5/12$.
\end{remark}

The remaining of the section is devoted to giving a proof of Theorem~\ref{theorem:IHES}, hence we always work in the setting of Theorem~\ref{theorem:IHES}.
We begin with an explanation of the construction of the models, their symmetries, and the generality condition on them.
We then present the proof of the main claim about their \Kstability; the proof uses the theory of refinements introduced in~\cite{AbbanZhuang}, which provides a lower bound for stability thresholds. Indeed, some formulae are readily available in~\cite[\S1.7]{calabi_problem_3folds} which compute the refinement for a flag, if the flag is geometrically realised on a Mori Dream Space birational model of the Fano variety. We will use them in the proof of Theorem~\ref{theorem:IHES}\eqref{item:IHES_3}.

\subsection{The models and their symmetries} \label{sec:preliminaries_on_del_Pezzo_surface_degree4}

\begin{proof}[{Proof of Theorem~\ref{theorem:IHES}\eqref{item:IHES_1}}]
All of this is very classical.
\end{proof}

Conversely, if $S$ is a smooth del Pezzo surface of degree $4$, then we can view $S$ as the complete intersection $\{g=h=0\}\subset \mathbb{P}^4$ where $g$ and $h$ are quadrics in $x_0, \dots, x_4$.
We denote by $\Sigma$ the set consisting of the 40 intersection points of the lines on $S$.

\begin{comment}
after a suitable linear change of coordinates, we can view $S$ as the complete intersection $\{g=h=0\}\subset \mathbb{P}^4$ with
\begin{equation} \label{eq:quadrics_defining_dP4}
	\begin{aligned}
g&=x_0^2+x_1^2+x_2^2+x_3^2+x_4^2, \\
h&=\epsilon_0x_0^2+\epsilon_1x_1^2+\epsilon_2x_2^2+\epsilon_3x_3^2+\epsilon_4x_4^2
	\end{aligned}
\end{equation}
and $\epsilon_0,\epsilon_1,\epsilon_2,\epsilon_3,\epsilon_4$ pairwise distinct complex numbers. We have denoted the coordinates on $\PP^4$ by $x_0,\dots,x_4$. Recall that the surface $S$ contains exactly $16$ lines, and their intersection points form a finite subset $\Sigma\subset S$ consisting of $40$ points.
\end{comment}

\begin{remark}
\label{remark:conics}
Let $p$ be a point in $S$.
\begin{enumerate}
	\item If $p$ is not contained in any line in $S$ then there are exactly ten smooth conics in $S$ that contain $p$;
	moreover, five of these ten conics can be chosen to intersect pairwise transversally at $p$.
	\item If $p$ is contained in a line in $S$ and $p\notin \Sigma$, then
	there are exactly five smooth conics in $S$ that contain $p$, and any two of them intersect transversally at $p$.
	\item If $p\in\Sigma$, then it is contained in a unique smooth conic in $S$.
\end{enumerate}
\end{remark}

%Let $\Delta$ be a smooth curve in $|-4K_S|$.
%Then $\Delta$ is cut out on $S$ by a quartic $\{f=0\}\subset\mathbb{P}^4$ for some homogeneous quartic polynomial $f\in\mathbb{C}[x_0,x_1,x_2,x_3,x_4]$.

\begin{proof}[{Proof of Theorem~\ref{theorem:IHES}\eqref{item:IHES_2}}]
We study the singularities of $X$ by looking at the affine charts of $\PP(1^5, 2^2)$.

Let us consider the chart $x_0 \neq 0$; this is isomorphic to $\AA^6$ with affine coordinates $x_1, \dots, x_4$, $y_0$, $y_1$.
Inside this $\AA^6$ the variety $X$ is given by the equations
\begin{align*}
y_0 y_1 - \bar{f} &= 0 \\
\bar{g} &= 0 \\
\bar{h} &= 0
\end{align*}
where $\bar{f} = f \vert_{x_0 = 1}$ and similarly for $\bar{g}$ and $\bar{h}$.
The jacobian matrix of the equations of $X$ in $\AA^6$ is
\[
J = \begin{pmatrix}
	- \partial_{x_1} \bar{f} & - \partial_{x_2} \bar{f} & - \partial_{x_3} \bar{f} & - \partial_{x_4} \bar{f} & y_1 & y_0 \\
	 \partial_{x_1} \bar{g} & \partial_{x_2} \bar{g} & \partial_{x_3} \bar{g} & \partial_{x_4} \bar{g} & 0 & 0 \\
	 	 \partial_{x_1} \bar{h} & \partial_{x_2} \bar{h} & \partial_{x_3} \bar{h} & \partial_{x_4} \bar{h} & 0 & 0
\end{pmatrix}.
\]
Since $S \subset \PP^4$ is a smooth surface, the bottom-left $2 \times 4$ submatrix of $J$ has rank $2$ in all points of $X$.
Therefore it is clear that $J$ has rank $3$ at all points of $X$ which satisfy $y_0 \neq 0$ or $y_1 \neq 0$.
The points of $X$ which satisfy $y_0 = y_1 = 0$ also satisfy $f=0$, and from the smoothness of $\Delta$ it follows that the rank of $J$ is $3$ at these points.
Therefore the intersection of $X$ with the chart $x_0 \neq 0$ is smooth.
Similarly, this holds also for the other charts $x_i \neq 0$ for $i \in \{ 0, \dots, 4\}$. In other words, we have proven that the singular points of $X$ must satisfy $x_0 = \cdots = x_4 = 0$. There are exactly two such points:
\begin{equation*}
p_0=[0:0:0:0:0:1:0] \quad \text{and} \quad p_1= [0:0:0:0:0:0:1].
\end{equation*}
We need to study the singularity type of these two points on $X$.

The chart $y_0 \neq 0$ on $\PP(1^5, 2^2)$ gives the quotient singularity $\frac{1}{2}(1,1,1,1,1,0)$ with orbifold coordinates $x_0, \dots, x_4, y_1$.
It is easy to see that $X \cap \{ y_0 \neq 0 \}$ is isomorphic to $\{ g=h = 0 \} \subset \frac{1}{2}(1,1,1,1,1)$. Since $\frac{1}{2}(1,1,1,1,1)$ is the affine cone over $( \PP^4, \cO(2))$, we have that $\{ g=h = 0 \} \subset \frac{1}{2}(1,1,1,1,1)$ is the affine cone over $(S, -2K_S)$, which is a klt singularity.
This shows that $p_0$ and $p_1$ are klt singularities of $X$.

By adjunction $-K_X = \cO_{\PP(1^5,2^2)} (1) \vert_X$, so $X$ is a klt Fano \threefold. Its degree is $(-K_X)^3 = \frac{1^3 \cdot 4 \cdot 2 \cdot 2}{2^2} = 4$.

Now we need to construct an effective action of $G = \CC^* \rtimes \ZZ / 2\ZZ$ on $X$. Consider the $\mathbb{C}^\ast$-action given by
\[
\left[ x_0:x_1:x_2:x_3:x_4:y_0:y_1 \right] \mapsto \left[ x_0:x_1:x_2:x_3:x_4:\lambda y_0: \lambda^{-1} y_1 \right]
\]
for $\lambda\in \mathbb{C}^\ast$. Furthermore, the group $\mathrm{Aut}(X)$ also contains an involution $\sigma$ that is given by
\[
\left[ x_0:x_1:x_2:x_3:x_4:y_0:y_1 \right] \mapsto \left[ x_0:x_1:x_2:x_3:x_4:y_1:y_0 \right].
\]
Together, they generate a subgroup $G\subset \mathrm{Aut}(X)$ that is isomorphic to $\mathbb{C}^\ast\rtimes\mathbb{Z}/2\mathbb{Z}$. Note that the two singular points $p_0$ and $p_1$ on $X$ are swapped by the action of $\sigma$ and fixed by the action of $\CC^\ast$.
This concludes the proof of Theorem~\ref{theorem:IHES}\eqref{item:IHES_2}.
\end{proof}

Let $\rho\colon X\dasharrow S$ be the rational map given by
\[
\left[x_0:x_1:x_2:x_3:x_4:y_0:y_1\right] \mapsto \left[ x_0:x_1:x_2:x_3:x_4 \right].
\]
Then $\rho$ is undefined precisely at the points $p_0$ and $p_1$, resulting in the following $G$-equivariant commutative diagram:
\begin{equation}
\label{equation:small-diagram}
\xymatrix{
&&\widetilde{X}\ar@{->}[dll]_{\pi}\ar@{->}[rrd]^{\eta}&&\\%
X\ar@{-->}[rrrr]^{\rho}&&&&S}
\end{equation}
where $\widetilde{X}$ is a smooth projective \threefold,
$\pi$ is a birational morphism that contracts two irreducible smooth surfaces $E_0$ and $E_1$ to the points $p_0$ and $p_1$, respectively,
and $\eta$ is a $G$-equivariant conic bundle. Furthermore, the surfaces $E_0$ and $E_1$ are sections of the conic bundle $\eta$,
so that $\eta$ induces isomorphisms $E_0\cong S$ and $E_1\cong S$. We have $E_0\vert_{E_0}\sim -2K_{E_0}$ and $E_1\vert_{E_1}\sim -2K_{E_1}$,
which gives
\[
-K_{\widetilde{X}}\sim_{\mathbb{Q}}\pi^*(-K_X)+\frac{1}{2}(E_0+E_1)\sim_{\mathbb{Q}}\eta^*(-K_S)+E_0+E_1.
\]

\begin{cor}
\label{corollary:MDS}
The \threefold{} $\widetilde{X}$ is a Mori Dream Space.
\end{cor}

\begin{proof}
Let $\varepsilon$ be a sufficiently small positive rational number.
Then $\left( \widetilde{X},\frac{1+\varepsilon}{2}(E_0+E_1) \right)$ has Kawamata log terminal singularities,
and $-(K_{\widetilde{X}}+\frac{1+\varepsilon}{2}(E_0+E_1))$ is ample.
The claim follows from~\cite[Corollary~1.3.1]{BCHM}.\end{proof}

Let $R_0$ and $R_1$ be the surfaces in $X$ that are cut out by $y_0=f=0$ and $y_1=f=0$, respectively.
Denote by $\widetilde{R}_0$ and $\widetilde{R}_1$ the strict transforms of $R_0$ and $R_1$ on $\widetilde{X}$, respectively.
Then $\eta^*(\Delta)=\widetilde{R}_0+\widetilde{R}_1$,
and $\Delta$ is the discriminant curve of the conic bundle $\eta$.
We have the following non-$G$-equivariant commutative diagram
\begin{equation}
\label{equation:big-diagram}
\xymatrix{
%&&&&&&&&\\%
&&U\ar@{->}[ddrr]_{\phi}\ar@{->}[dll]_{\psi}&&\widetilde{X}\ar@{->}[dd]^{\eta}\ar@{->}[rr]^{\eta_1}\ar@{->}[ll]_{\eta_0}&&U\ar@{->}[ddll]^{\phi}\ar@{->}[drr]^{\psi}&&\\%
V\ar@{-->}[drrrr]_{\tau}&&&& &&&&V\ar@{-->}[dllll]^{\tau}\\
&&&&S&&&&}
\end{equation}
where $U=\mathbb{P}(\mathcal{O}_S(2K_S)\oplus\mathcal{O}_S)$,
$\phi$ is the $\mathbb{P}^1$-bundle given by the projection $\mathbb{P}(\mathcal{O}_S(2K_S)\oplus\mathcal{O}_S)\to S$,
the morphism $\psi$ is a contraction of the negative section of the $\mathbb{P}^1$-bundle $\phi$,
$V$ is a cone over $S$, the map $\tau$ is the projection from the vertex of the cone,
and morphisms $\eta_0$ and $\eta_1$ are birational contractions of the surfaces $\widetilde{R}_0$ and $\widetilde{R}_1$, respectively.

\begin{remark}
\label{remark:big-diagram}
In the left hand side of~\eqref{equation:big-diagram}, we have $\phi^*(\Delta)=\eta_0(R_1)$,
the morphism $\psi$ contracts $\eta_0(E_0)$, and $\psi\circ\eta_0(E_1)$ is a smooth hyperplane section of the cone $V\subset\mathbb{P}^{13}$.
In the right hand side of~\eqref{equation:big-diagram}, we have $\phi^*(\Delta)=\eta_1(R_0)$,
$\psi$ contracts $\eta_1(E_1)$, and $\psi\circ\eta_1(E_0)$ is a hyperplane section of the cone $V$.
The involution $\sigma$ swaps the left and the right hand sides of the diagram~\eqref{equation:big-diagram}.
\end{remark}

Set $Z=R_0\cap R_1$ and $\widetilde{Z}=\widetilde{R}_0\cap\widetilde{R}_1$.
Then $Z=\pi(\widetilde{Z})$, $Z\cong\widetilde{Z}\cong \Delta$, and $Z=\{y_0=y_1=f=0\}\subset X$.
Observe that $p_0\not\in Z$, $p_1\not\in Z$, $\widetilde{Z}\cap E_0=\widetilde{Z}\cap E_1=\varnothing$,
and every point in $Z$ and $\widetilde{Z}$ is $G$-invariant.

\begin{lemma}
\label{lemma:G-fixed-points}
Let $q$ be a $G$-invariant point in $\widetilde{X}$, let $C$ be a $G$-invariant irreducible curve in $\widetilde{X}$,
and let $B$ be a $G$-invariant irreducible surface in $\widetilde{X}$.
Then the following assertions hold:
\begin{enumerate}
\item $q\in\widetilde{Z}$;
\item either $C=\widetilde{Z}$ or $C$ is a smooth fibre of the conic bundle $\eta$;
\item $B=\eta^*(\mathscr{C})$ for some irreducible curve $\mathscr{C}\subset S$.
\end{enumerate}
In particular, we have that $B\cap\widetilde{Z}\ne\varnothing$.
\end{lemma}

\begin{proof}
Left to the~reader.
\end{proof}

\subsection{\Kstability{} of the general models}

The proof of Theorem~\ref{theorem:IHES}\eqref{item:IHES_3} is quite involved. However, it is less difficult to produce specific examples of $X$ with large symmetries and prove they are \Kpolystable{} by taking advantage of the group of symmetries. Here, for illutration, we present one such case. Then we proceed by proving that a general $X$ is \Kpolystable{} (see Theorem~\ref{theorem:IHES}\eqref{item:IHES_3}).

\begin{proof}[{Special case of Theorem~\ref{theorem:IHES}\eqref{item:IHES_3}: diagonal models}]
	Suppose that
	\begin{align*}
		f&=\alpha_0x_0^4+\alpha_1x_1^4+\alpha_2x_2^4+\alpha_3x_3^4+\alpha_4x_0^4, \\
			g&=x_0^2+x_1^2+x_2^2+x_3^2+x_4^2, \\
		h&=\epsilon_0x_0^2+\epsilon_1x_1^2+\epsilon_2x_2^2+\epsilon_3x_3^2+\epsilon_4x_4^2,
	\end{align*}
where $\alpha_0,\dots,\alpha_4$ are sufficiently general complex numbers and $\epsilon_0, \dots, \epsilon_4$ are pairwise distinct complex numbers.
%Then the curve $\Delta$ is smooth.
Then the group $\mathrm{Aut}(X)$ is larger than the group $G$ in Theorem~\ref{theorem:IHES}\eqref{item:IHES_2}, because $\mathrm{Aut}(X)$ also contains
$15$ additional involutions given by
\[
\left[x_0:x_1:x_2:x_3:x_4:y_0:y_1\right] \mapsto \left[x_0:(-1)^ax_1:(-1)^b x_2:(-1)^c x_3:(-1)^dx_4:y_0:y_1 \right]
\]
for $a,b,c,d\in\{0,1\}$.
Let $\Gamma$ be the subgroup in $\mathrm{Aut}(X)$ generated by $G$ and these $15$ involutions.
Then $\Gamma\cong\mathbb{C}^\ast\rtimes(\mathbb{Z}/2\mathbb{Z})^5$, and~\eqref{equation:small-diagram} is $\Gamma$-equivariant.
Note that $X$ does not contain $\Gamma$-invariant points.
Furthermore, using the diagram~\eqref{equation:small-diagram} and~\cite[Theorem~6.9]{DolgachevIskovskikh}, we conclude that $\mathrm{Cl}^G(X)=\mathbb{Z}[-K_X]$.
Now, arguing as in the proof of~\cite[Theorem~1.52]{calabi_problem_3folds}, we conclude that
\[
\alpha_\Gamma (X) \geq 1,
\]
where $\alpha_\Gamma(X)$ is the $\Gamma$-invariant $\alpha$-invariant of Tian for the \threefold{} $X$, see~\cite{Tian-alpha} for the analytic definition and~\cite{CS-alpha} for an algebraic definition of $\alpha$.
It follows that $X$ is \Kpolystable{} by~\cite{Tian-alpha} as \[\alpha_\Gamma\big(X\big)>\dfrac{3}{4}=\dfrac{\dim X}{1+\dim X}.\]
This concludes the claim that $X$ is \Kpolystable.
\end{proof}

\begin{proof}[{Proof of Theorem~\ref{theorem:IHES}\eqref{item:IHES_3}}]
Suppose that $X$ is not \Kpolystable.
By~\cite[Corollary~4.14]{Zhuang}, there exists a $G$-invariant prime divisor $\mathbf{F}$ over $X$ for which $\beta(\mathbf{F})\leqslant 0$.
For the definition of $\beta(\mathbf{F})$, see~\cite{Fujita-beta,Li-beta}. Let $\mathfrak{C}$ be the center of the divisor $\mathbf{F}$ on $X$. Then, for every point $P\in \mathfrak{C}$, we must have $\delta_P(X)\leqslant 1$,
where $\delta_P(X)$ is the local stability threshold of $X$ at $p$. For the~precise definition of $\delta_P(X)$, see~\cite[Definition 2.5]{AbbanZhuang} and~\cite[\S~1.5]{calabi_problem_3folds}. On the other hand, by Lemma~\ref{lemma:G-fixed-points}, we have the following four possibilities:
\begin{enumerate}
\item $\mathfrak{C}$ is a $G$-invariant point in $Z$;
\item $\mathfrak{C}$ is the curve $Z$;
\item $\mathfrak{C}$ is an irreducible fibre of the~rational map $\rho\colon X\dasharrow S$;
\item $\mathfrak{C}$ is a $G$-invariant surface and $\mathfrak{C}\cap Z\ne\varnothing$.
\end{enumerate}
Hence, we see that either $\mathfrak{C}$ is an irreducible curve that is a fibre of the~rational map $\rho\colon X\dasharrow S$,
or the curve $Z$ contains a point $x$ such that $\delta_x(X)\leqslant 1$. We will show that both cases are impossible, by first lifting the problem to the level of $\widetilde{X}$, then using the results obtained in~\cite{AbbanZhuang} and~\cite[\S~1.7]{calabi_problem_3folds} to derive a contradiction.

Since $\widetilde{X}$ is smooth and it is a Mori Dream Space, it is more convenient to work on~$\widetilde{X}$ than on~$X$.
Namely, we let $L=\pi^*(-K_X)$, and we let $\widetilde{\mathfrak{C}}$ be the~center of the~divisor~$\mathbf{F}$ on the \threefold{} $\widetilde{X}$.
Then $\widetilde{\mathfrak{C}}$ is the strict transform of the center $\mathfrak{C}$,
and we have the following four possibilities:
\begin{enumerate}
\item $\widetilde{\mathfrak{C}}$ is a $G$-invariant point in $\widetilde{Z}$;
\item $\widetilde{\mathfrak{C}}$ is the curve $\widetilde{Z}$;
\item $\widetilde{\mathfrak{C}}$ is a smooth fibre of the~conic bundle $\eta\colon\widetilde{X}\to S$;
\item $\widetilde{\mathfrak{C}}$ is a $G$-invariant surface and $\widetilde{\mathfrak{C}}\cap\widetilde{Z}\ne\varnothing$.
\end{enumerate}
Note that $\delta_x(\widetilde{X},L)=\delta_{\pi(x)}(X)$ for every point $x\in\widetilde{Z}$.
For the definition of $\delta_x(\widetilde{X},L)$, see~\cite[Definition 2.5]{AbbanZhuang} and~\cite[\S~1.5]{calabi_problem_3folds}.
Therefore, we conclude that one of the following two cases hold:
\begin{enumerate}
\item[($\diamondsuit$)] $\widetilde{\mathfrak{C}}$ a smooth fibre of the~conic bundle $\eta\colon\widetilde{X}\to S$,
\item[($\heartsuit$)] there exists a point $x\in\widetilde{Z}$ such that $\delta_x(\widetilde{X},L)\leqslant 1$.
\end{enumerate}
In both cases, let us introduce a new curve $\mathcal{C}$ in the \threefold{} $\widetilde{X}$ as follows:
\begin{itemize}
\item in case ($\diamondsuit$), we let $\mathcal{C}=\widetilde{\mathfrak{C}}$,
\item in case ($\heartsuit$), we let $\mathcal{C}$ be the (singular) fibre of the conic bundle $\eta$ that contains $x$.
\end{itemize}

Let $p=\eta(\mathcal{C})$. It follows from Remark~\ref{remark:conics} that $S$
contains a smooth conic $\mathscr{C}$ such that $p\in\mathscr{C}$.
Moreover, if $p\in\Delta$, it follows from Remark~\ref{remark:conics} that we can also choose $\mathscr{C}$ such that it intersects the curve $\Delta$
transversally at the point $p$, because $\Sigma\cap\Delta=\varnothing$ by assumption.

Let $B=\eta^*(\mathscr{C})$, and let $B^\prime=\eta^*(Z)$ for a general conic $Z\in |-K_S-\mathscr{C}|$.
Then $B$ is normal, it has at most Du Val singularities of type $\mathbb{A}$,
it is smooth along the curve $\mathcal{C}$, and $B^\prime$ is smooth.

Let us compute $\beta(B)$. We have $\beta(B)=A_X(B)-S_L(B)=1-S_L(B)$, where
\[
S_L(B)=\frac{1}{L^3}\int_{0}^\infty\mathrm{vol}\big(L-uB\big)du.
\]
For $u\in\mathbb{R}_{\geqslant 0}$, the divisor $L-uB$ is pseudo-effective $\iff$ $u \leqslant 1$, because
\[
L-uB\sim_{\mathbb{Q}}(1-u)B+B^\prime+\frac{1}{2}\big(E_0+E_1\big).
\]
For $u\in[0,1]$, let $P(u)$ be the positive part of the Zariski decomposition of the divisor $L-uB$,
and let $N(u)$ be the negative part of the Zariski decomposition of the divisor $L-uB$.
Then
\[
P(u)=L-uB-\frac{u}{2}\big(E_0+E_1\big)
\]
and $N(u)=\frac{u}{2}(E_0+E_1)$ for every $u\in[0,1]$. This gives
\[
S_L(B)=\frac{1}{L^3}\int_{0}^\infty\mathrm{vol}\big(L-uB\big)du=\frac{1}{4}\int_{0}^1\big(P(u)\big)^3du=\frac{1}{4}\int_{0}^1(2u^3-6u+4)du=\frac{3}{8},
\]
which implies that $\beta(B)=\frac{5}{8}$. Here, we used the following intersections on the \threefold{} $\widetilde{X}$:
\begin{center}
$B^3=0$, $(B^\prime)^3=0$, $E_0^3=16$, $E_1^3=16$, $E_0^2\cdot E_1=0$, $E_0\cdot E_1^2=0$, $E_0\cdot B\cdot E_1=0$, \\
$E_0\cdot B^\prime\cdot E_1=0$, $E_0\cdot B\cdot B^\prime=2$, $B\cdot B^\prime\cdot E_1=2$, $E_0\cdot B^2=0$, $B^2\cdot B^\prime=0$, $E_0\cdot (B^\prime)^2=0$, $B\cdot (B^\prime)^2=0$, \\
$B^2\cdot E_1=0$, $E_0^2\cdot B=-4$, $B\cdot E_1^2=-4$, $E_0^2\cdot B^\prime=-4$, $B^\prime\cdot E_1^2=-4$, $(B^\prime)^2\cdot E_1=0$.
\end{center}

We set $\mathbf{e}_0=E_0\vert_{B}$ and $\mathbf{e}_1=E_1\vert_{B}$.
Then~$\mathbf{e}_0$ and $\mathbf{e}_1$ are smooth irreducible rational disjoint curves,
and $B$ is smooth along $\mathbf{e}_0$ and $\mathbf{e}_1$.
On $B$, we have $\mathbf{e}_0^2=\mathbf{e}_1^2=-4$, $\mathbf{e}_0\cdot\mathbf{e}_1=0$, $\mathbf{e}_0\cdot\mathcal{C}=\mathbf{e}_1\cdot\mathcal{C}=1$.

Suppose that case ($\diamondsuit$) holds.
Then $\widetilde{\mathfrak{C}}$ is a smooth fibre of the~conic bundle $\eta$, and $\mathcal{C}=\widetilde{\mathfrak{C}}$. Set
\[
S\big(W^B_{\bullet,\bullet};\mathcal{C}\big)=\frac{3}{L^3}\int_0^1\int_0^{\infty} \mathrm{vol}\big(P(u)\big\vert_{B}-v\mathcal{C}\big)dvdu.
\]
Using~\cite[Corollary~1.7.26]{calabi_problem_3folds}, we get $S(W^B_{\bullet,\bullet};\mathcal{C})\geqslant 1$,
since $\beta(\mathbf{F})\leqslant 0$, $\beta(B)>0$ and $\mathcal{C}\not\subset\mathrm{Supp}(N(u))$.
On the other hand, it is easy to compute $S(W^B_{\bullet,\bullet};\mathcal{C})$. Indeed, take $v\in\mathbb{R}_{\geqslant 0}$. Then
\[
P(u)\big\vert_{B}-v\mathcal{C}\sim_{\mathbb{R}} (2-v)\mathcal{C}+\frac{1-u}{2}\big(\mathbf{e}_0+\mathbf{e}_1\big).
\]
Therefore, the divisor $P(u)\vert_{B}-v\mathcal{C}$ is nef for $v\leqslant 2u$, and it is not pseudo-effective for $v>2$.
Moreover, if $v\in[2u,2]$, the positive part of its Zariski decomposition is $\frac{2-v}{4}(4\mathcal{C}+\mathbf{e}_0+\mathbf{e}_1)$.
Then
\[
\mathrm{vol}\big(P(u)\vert_{B}-vC\big)=
\left\{\aligned
&2-2u^2+2uv-2v\ \text{if $0\leqslant v\leqslant 2u$}, \\
&\frac{(v-2)^2}{2}\ \text{if $2u\leqslant v\leqslant 2$}, \\
&0\ \text{if $v\geqslant 2$}.
\endaligned
\right.
\]
Integrating, we get $S(W^B_{\bullet,\bullet};\mathcal{C})=\frac{3}{4}$, which is a contradiction.
Thus, case ($\diamondsuit$) does not hold.

Now, we assume that the case ($\heartsuit$) holds, and derive a contradiction.
Recall that this means that $\delta_x(\widetilde{X},L)\leqslant 1$ for some point $x\in\widetilde{Z}$.
In this case, the curve $\mathcal{C}$ is singular.
Namely, we have $\mathcal{C}=C_0+C_1$, where $C_0$ and $C_1$ are smooth irreducible rational curves that intersect each other transversally at $x$.
Without loss of generality, we may assume that $C_0\cap E_0\ne\varnothing$.
Moreover, since the surface $B$ is smooth along $\mathcal{C}$,
the numerical intersections of the curves $C_0$, $C_1$, $\mathbf{e}_0$, $\mathbf{e}_1$ on the surface $B$ are given in the following table:
\begin{center}
\renewcommand\arraystretch{1.4}
\begin{tabular}{|c||c|c|c|c|}
\hline & $C_0$ & $C_1$ & $\mathbf{e}_0$ & $\mathbf{e}_1$ \\
\hline
\hline
$C_0$ & $-1$ & $1$& $1$&$0$\\
\hline
$C_1$ & $1$ & $-1$ & $0$ & $1$\\
\hline
$\mathbf{e}_0$ & $1$ & $0$ & $-4$ & $0$\\
\hline
$\mathbf{e}_1$ & $0$ & $1$ & $0$ & $-4$\\
\hline
\end{tabular}
\end{center}

Let us estimate $\delta_x(\widetilde{X},L)$ using~\cite[Theorem~1.7.30]{calabi_problem_3folds}.
We have
$N(u)\vert_B=\frac{u}{2}(\mathbf{e}_0+\mathbf{e}_1)$ and $x\not\in\mathbf{e}_0\cup\mathbf{e}_1$.
For every $u\in [0,1]$, we let
\[
t(u)=\inf\Big\{v\in \mathbb R_{\geqslant 0} \ \big|\ \text{the divisor $P(u)\big|_B-vC_0$ is pseudo-effective}\Big\}.
\]
For $v\in [0, t(u)]$, we let $P(u,v)$ be the~positive part of the Zariski decomposition of $P(u)|_B-vC_0$, and we let $N(u,v)$ be its negative part.
Then we set
\[
S\big(W^B_{\bullet,\bullet};C_0\big)=\frac{3}{L^3}\int_0^1\int_0^{t(u)} \mathrm{vol}\big(P(u)\big\vert_{B}-vC_0\big)dvdu.
\]
Note that $C_0\not\subset \mathrm{Supp}(N(u,v))$ for every $u\in [0,1)$ and $v\in (0,t(u))$.
Thus, we can let
\[
F_x\big(W_{\bullet,\bullet,\bullet}^{B,C_0}\big)=\frac{6}{L^3} \int_0^1\int_0^{t(u)}\big(P(u,v)\cdot C_0\big)\cdot \mathrm{ord}_x\big(N(u,v)\big|_{C_0}\big)dvdu.
\]
Finally, we let
\[
S\big(W_{\bullet, \bullet,\bullet}^{B,C_0};x\big)=\frac{3}{L^3}\int_0^1\int_0^{t(u)}\big(P(u,v)\cdot C_0\big)^2dvdu+F_x\big(W_{\bullet,\bullet,\bullet}^{B,C_0}\big).
\]
Then, since $C_0\not\subset\mathrm{Supp}(N(u))$, it follows from~\cite[Theorem~1.7.30]{calabi_problem_3folds} that
\[
1\geqslant\delta_{x}(\widetilde{X},L)\geqslant \min\left\{\frac{1}{S(W_{\bullet, \bullet,\bullet}^{B,C_0}; x)}, \frac{1}{S(W_{\bullet,\bullet}^B;C_0)},\frac{1}{S_L(B)}\right\}.
\]
Recall that $S_L(B)=\frac{3}{8}$. Thus, either $S(W_{\bullet, \bullet,\bullet}^{B,C_0}; x)\geqslant 1$ or $S(W_{\bullet,\bullet}^B;C_0)\geqslant 1$ (or both).

Let us compute $S(W_{\bullet, \bullet,\bullet}^{B,C_0}; x)$ and $S(W_{\bullet,\bullet}^B;C_0)$.
As above, take $v\in\mathbb{R}_{\geqslant 0}$. Then
\[
P(u)\big\vert_{B}-vC_0\sim_{\mathbb{R}} (2-v)C_0+2C_1+\frac{1-u}{2}\big(\mathbf{e}_0+\mathbf{e}_1\big).
\]
Therefore, since the intersection form of the curves $C_1$, $\mathbf{e}_0$, $\mathbf{e}_1$ on the surface $B$ is negative definite,
we see that $P(u)\vert_{B}-vC_0$ is pseudoeffective $\iff$ $v\leqslant 2$, so $t(u)=2$.
Moreover, if $0\leqslant u\leqslant\frac{1}{5}$, then
\[
%\hspace{2.2cm}
P(u,v)=\left\{\arraycolsep=8pt\def\arraystretch{1.6}\begin{array}{ll}
(2-v)C_0+2C_1+\frac{1}{2}(1-u)\big(\mathbf{e}_0+\mathbf{e}_1\big) & \text{if $0\leqslant v\leqslant 2u$}, \\
\frac{1}{4}(2-v)\big(4C_0+\mathbf{e}_0\big)+2C_1+\frac{1}{2}(1-u)\mathbf{e}_1 & \text{if $2u\leqslant v\leqslant \frac{1}{2}(1-u)$}, \\
\frac{1}{4}(2-v)\big(4C_0+\mathbf{e}_0\big)+\frac{1}{2}(5-2v-u)C_1+\frac{1}{2}(1-u)\mathbf{e}_1 & \text{if $1-u\leqslant 2v\leqslant 1+3u$}, \\
\frac{1}{12}(2-v)(12C_0 + 16C_1+3\mathbf{e}_0+4\mathbf{e}_1) & \text{if $\frac{1}{2}(1+3u)\leqslant v\leqslant 2$},
\end{array}
\right.
\]
and,
\[
%\hspace{1.6cm}
N(u,v)=\left\{\arraycolsep=8pt\def\arraystretch{1.6}\begin{array}{ll}
0& \text{if $0\leqslant v\leqslant 2u$}, \\
\frac{1}{4}(v-2u)\mathbf{e}_0 &\text{if $2u\leqslant v\leqslant \frac{1}{2}(1-u)$}, \\
\frac{1}{4}(v-2u)\mathbf{e}_0+\frac{1}{2}(2v+u-1)C_1 &\text{if $1-u\leqslant 2v\leqslant 1+3u$}, \\
\frac{1}{4}(v-2u)\mathbf{e}_0+\frac{1}{3}(4v-2)C_1+\frac{1}{6}(2v-3u-1)\mathbf{e}_1 &\text{if $\frac{1}{2}(1+3u)\leqslant v\leqslant 2$},
\end{array}
\right.
\]
which compute
\[
\hspace{-4cm}
P(u,v)\cdot C_0=\left\{\arraycolsep=8pt\def\arraystretch{1.6}\begin{array}{ll}
\frac{1}{2}(1-u+2v) & \text{if $0\leqslant v\leqslant 2u$}, \\
\frac{1}{4}(2+3v) & \text{if $2u\leqslant v\leqslant \frac{1}{2}(1-u)$}, \\
\frac{1}{4}(4-2u-v)& \text{if $1-u\leqslant 2v\leqslant 1+3u$}, \\
\frac{1}{12}(14-7v) & \text{if $\frac{1+3u}{2}\leqslant v\leqslant 2$},
\end{array}
\right.
\]
\[
\hspace{-2.2cm}
\mathrm{vol}\big(P(u)\vert_{B}-vC_0\big)=\left\{\arraycolsep=8pt\def\arraystretch{1.6}\begin{array}{ll}
2-2u^2+uv-v^2-v& \text{if $0\leqslant v\leqslant 2u$}, \\
2-v-u^2-\frac{3}{4}v^2& \text{if $2u\leqslant v\leqslant \frac{1}{2}(1-u)$}, \\
\frac{1}{4}(9-3u^2+4uv+v^2-2u-8v)& \text{if $1-u\leqslant 2v\leqslant 1+3u$}, \\
\frac{7}{12}(2-v)^2& \text{if $\frac{1}{2}(1+3u)\leqslant v\leqslant 2$}.
\end{array}
\right.
\]
Similarly, if $\frac{1}{5}\leqslant u\leqslant 1$, then
\[
\hspace{1cm}
P(u,v)=\left\{\arraycolsep=8pt\def\arraystretch{1.6}\begin{array}{ll}
(2-v)C_0+2C_1+\frac{1}{2}(1-u)\big(\mathbf{e}_0+\mathbf{e}_1\big)& \text{if $0\leqslant v\leqslant \frac{1}{2}(1-u)$}, \\
(2-v)C_0+\frac{1}{2}(5-2v-u)C_1+\frac{1}{2}(1-u)\big(\mathbf{e}_0+\mathbf{e}_1\big)& \text{if $\frac{1}{2}(1-u)\leqslant v\leqslant 2u$}, \\
\frac{1}{4}(2-v)\big(4C_0+\mathbf{e}_0\big)+\frac{1}{2}(5-2v-u)C_1+\frac{1}{2}(1-u)\mathbf{e}_1& \text{if $1-u\leqslant 2v\leqslant 1+3u$}, \\
\frac{1}{12}(2-v)(12C_0 + 16C_1+3\mathbf{e}_0+4\mathbf{e}_1)& \text{if $\frac{1}{2}(1+3u)\leqslant v\leqslant 2$},
\end{array}
\right.
\]
\[\hspace{0.2cm}
N(u,v)=\left\{\arraycolsep=8pt\def\arraystretch{1.6}\begin{array}{ll}
0& \text{if $0\leqslant v\leqslant \frac{1}{2}(1-u)$}, \\
\frac{1}{2}(2v+u-1)C_1& \text{if $\frac{1}{2}(1-u)\leqslant v\leqslant 2u$}, \\
\frac{1}{4}(v-2u)\mathbf{e}_0+\frac{1}{2}(2v+u-1)C_1& \text{if $1-u\leqslant 2v\leqslant 1+3u$}, \\
\frac{1}{4}(v-2u)\mathbf{e}_0+\frac{1}{3}(4v-2)C_1+\frac{1}{6}(2v-3u-1)\mathbf{e}_1& \text{if $\frac{1}{2}(1+3u)\leqslant v\leqslant 2$},
\end{array}
\right.
\]

\[\hspace{-5.4cm}
P(u,v)\cdot C_0=\left\{\arraycolsep=8pt\def\arraystretch{1.6}\begin{array}{ll}
\frac{1}{2}(1-u+2v)& \text{if $0\leqslant v\leqslant \frac{1}{2}(1-u)$}, \\
1-u& \text{if $\frac{1}{2}(1-u)\leqslant v\leqslant 2u$}, \\
\frac{1}{4}(4-2u-v)& \text{if $1-u\leqslant 2v\leqslant 1+3u$}, \\
\frac{1}{12}(14-7v)& \text{if $\frac{1}{2}(1+3u)\leqslant v\leqslant 2$},
\end{array}
\right.
\]
\[
\mathrm{vol}\big(P(u)\vert_{B}-vC_0\big)=\left\{\arraycolsep=8pt\def\arraystretch{1.6}\begin{array}{ll}
2-2u^2+uv-v^2-v& \text{if $0\leqslant v\leqslant \frac{1}{2}(1-u)$}, \\
\frac{1}{4}(1-u)(7u-8v+9)& \text{if $\frac{1}{2}(1-u)\leqslant v\leqslant 2u$}, \\
\frac{1}{4}(9-3u^2+4uv+v^2-2u-8v)& \text{if $1-u\leqslant 2v\leqslant 1+3u$}, \\
\frac{7}{12}(2-v)^2& \text{if $\frac{1}{2}(1+3u)\leqslant v\leqslant 2$}.
\end{array}
\right.
\]

Now, integrating $\mathrm{vol}(P(u)\vert_{B}-vC_0)$ and $(P(u,v)\cdot C_0)^2$, we obtain $S(W_{\bullet,\bullet}^B;C_0)=\frac{13}{16}$ and
\[
S\big(W_{\bullet,\bullet,\bullet}^{B,C_0};x\big)=\frac{77}{320}+F_x\big(W_{\bullet,\bullet,\bullet}^{B,C_0}\big).
\]
Furthermore, since $x\not\in\mathbf{e}_0\cup\mathbf{e}_1$ and the curves $C_0$ and $C_1$ intersect transversally at $x$,
we have
%\begin{multline*}

\begin{align*}
F_x\big(W_{\bullet,\bullet,\bullet}^{B,C_0}\big)&=\frac{3}{2}\int_0^1\int_0^{2}\mathrm{ord}_x\big(N(u,v)\big|_{C_0}\big)\times \big(P(u,v)\cdot C_0\big)dvdu\\
&=\frac{3}{2}\int_0^{\frac{1}{5}}\int_{\frac{1-u}{2}}^{\frac{1+3u}{2}}\frac{2v+u-1}{2}\times\big(P(u,v)\cdot C_0\big)dvdu+
\frac{3}{2}\int_0^{\frac{1}{5}}\int_{\frac{1+3u}{2}}^{2}\frac{4v-2}{3}\times\big(P(u,v)\cdot C_0\big)dvdu\\
&+\frac{3}{2}\int_{\frac{1}{5}}^{1}\int_{\frac{1-u}{2}}^{\frac{1+3u}{2}}\frac{2v+u-1}{2}\times\big(P(u,v)\cdot C_0\big)dvdu+
\frac{3}{2}\int_{\frac{1}{5}}^{1}\int_{\frac{1+3u}{2}}^{2}\frac{4v-2}{3}\times\big(P(u,v)\cdot C_0\big)dvdu\\
&=\frac{3}{2}\int_0^{\frac{1}{5}}\int_{\frac{1-u}{2}}^{\frac{1+3u}{2}}\frac{2v+u-1}{2}\times\frac{4-2u-v}{4}dvdu+
\frac{3}{2}\int_0^{\frac{1}{5}}\int_{\frac{1+3u}{2}}^{2}\frac{4v-2}{3}\times\frac{14-7v}{12}dvdu\\
&+\frac{3}{2}\int_{\frac{1}{5}}^{1}\int_{\frac{1-u}{2}}^{2u}\frac{2v+u-1}{2}\times(1-u)dvdu+
\frac{3}{2}\int_{\frac{1}{5}}^{1}\int_{2u}^{\frac{1+3u}{2}}\frac{2v+u-1}{2}\times\frac{4-2u-v}{4}dvdu\\
&+\frac{3}{2}\int_{\frac{1}{5}}^{1}\int_{\frac{1+3u}{2}}^{2}\frac{4v-2}{3}\times\frac{14-7v}{12}dvdu=\frac{183}{320},
\end{align*}
%\quad \quad \quad \quad \quad\quad \quad \quad \quad \quad
%\end{multline*}
which gives $S(W_{\bullet,\bullet,\bullet}^{B,C_0};x)=\frac{13}{16}$. Thus, we see that $S(W_{\bullet, \bullet,\bullet}^{B,C_0}; x)=S(W_{\bullet,\bullet}^B;C_0)=\frac{13}{16}<1$, which is a contradiction.
This completes the proof of Theorem~\ref{theorem:IHES}.
\end{proof}

%The proof of Theorem~\ref{theorem:IHES} is similar to the proof of K-polystability of all smooth Fano \threefolds{}
%in the deformation families \textnumero 3.9 and \textnumero 4.2 presented in~\cite[\S~4.6]{calabi_problem_3folds}.
%Our proof can also be considered as a three\nobreakdash-dimensional counter-part of what is done in~\cite[\S~4.1]{KimViswanathanWon} and~\cite{LiuPetracci}.

\section{K-moduli compactification}\label{limits}
As in previous sections, let $X$ be a pure complete intersection $X_{2,2,4}\!\subset\! \PP(1^5, 2^2)$, i.e.\ a $(2,2,4)$-complete intersection in $\PP(1^5,2^2)$ which is not a $(2,4)$-complete intersection in $\PP(1^5,2)$.
For convenience,  assume at no cost that $X$ is given by $y_1^2= y_0^2-f(x_0,\cdots, x_4)$, $g(x_0,\cdots, x_4)=0$ and $h(x_0,\cdots,x_4)=0$ where $f$, $g$, and $h$ have degree $4$, $2$, and $2$ respectively. First, we make a slight improvement on the previous results using techniques of cyclic covers, cone construction, degeneration and interpolation.

\begin{thm}\label{ALL}
Let $\Delta = \{f=0, g=h=0\}\subset \bP^4$ and $S=\{g=h=0\}\subset \bP^4$. Assume $\Delta$ is a smooth curve and $S$ is a smooth surface. Then $X$ is K-polystable.
\end{thm}

\begin{proof}
By forgetting $y_1$ in the defining equation of $X$ we obtain a double cover $X\to (Y, \frac{1}{2}D)$ where $Y=\{g=h=0\}\subset \PP(1^5, 2)$ and $D =\{y_0^2=f\} |_Y$. In other words, $Y\cong C_p(S, \cO_S(2))$. By rescaling the $y_0$ coordinate only, it is easy to see that $(Y,D)$ admits an isotrivial degeneration to $(Y, D_0)$ where $D_0=\{y_0^2=0\}|_Y$ is twice of the section  at infinity, which we denote by $S_\infty = \frac{1}{2}D_0$. Since $S$ is a  smooth del Pezzo surface of degree $4$, we know that $S$ is K-stable. Hence by the cone construction \cite{LL16, LX20, ZhangZhou} we know that $(Y, (1-\frac{r}{3})S_\infty)$ is K-polystable for $\cO_S(2)=-r^{-1}K_{S}$, i.e. $r=\frac{1}{2}$. Hence $(Y, \frac{5}{12}D_0)$ is K-polystable which implies that $(Y, \frac{5}{12}D)$ is K-semistable by openness. Since $D\sim \cO_Y(4)$ and $-K_Y \sim \cO_Y(3)$, we know that $(Y, \frac{3}{4}D)$ is a klt log Calabi-Yau pair as $D$ is smooth and contained in the smooth locus of $Y$. Therefore by interpolation of K-stability \cite[Proposition 2.13]{wall_crossing_K_moduli} we know that $(Y, cD)$ is K-stable for every $c\in (\frac{5}{12},\frac{3}{4})$. In particular, taking $c=\frac{1}{2}$ yields K-stability of $(Y, \frac{1}{2}D)$ which implies K-polystability of $X$ by the cyclic cover result on K-polystability \cite{dervan_covers, LZ22, Zhuang}.
\end{proof}

Next we study the K-moduli compactification of pure $(2,2,4)$-complete intersections in $\PP(1^5,2^2)$. The following result is inspired by \cite[Theorem 5.2]{adl_quartic_K3}.

\begin{prop}\label{prop:surface-3fold}
Let $S$ be a log del Pezzo surface. Let  $\Delta \sim -4K_S$ be an effective $\bQ$-Cartier Weil divisor given by $\Delta = \{f=0\}$ for $f\in H^0(S, -4K_S)$. Let $Y=C_p(S, -2K_S)$ be the projective cone. Let $D= \{y^2 = f\}$ be a divisor in $Y$, where $\{y=0\}$ is the section at infinity of $Y$. Let $c\in (0,\frac{1}{4})\cap \bQ$. If $(S, c\Delta)$ is K-semistable (resp. K-polystable), then $(Y, \frac{5+16c}{12} D)$ is also K-semistable (resp. K-polystable).
\end{prop}

\begin{proof}
The proof is similar to \cite[Proof of Theorem 5.2]{adl_quartic_K3}. Denote by $a:=\frac{5+16c}{12}$ and $b:=\frac{11+4c}{12}$.
Denote by $\tau: Y\to Y$ the involution preserving each ruling given by $\tau^* y = -y$. Let $Z=C_p(S, -4K_S)$ be a new projective cone where $S_1= \{z=0\}$ is the section at infinity. Then the quotient map of $\tau$ yields a double cover $\pi: Y\to Z$ branched along $S_1$. Moreover, it is clear that $D$ is $\tau$-invariant whose quotient $S_2=D/\tau$ is a section of $Z$ such that $S_2|_{S_1} = \Delta$. Thus we have a crepant finite Galois morphism $\pi: (Y, aD)\to (Z, \frac{1}{2}S_1 + aS_2)$. By \cite{LZ22, Zhuang}, it suffices to show K-semistability (resp. K-polystability) for $(Z, \frac{1}{2}S_1 + aS_2)$.

We first treat K-semistability.
The natural $\bG_m$-action on $(Z, S_1)$ degenerates $S_2$ to $S_{2,0}$ as the cone over $\Delta$. Let $r\in \bQ_{>0}$ be chosen such that $-4K_S \sim_{\bQ} r^{-1} (-K_S - c\Delta)$, i.e.\ $r=\frac{1-4c}{4}$. It follows from the cone construction \cite[Proposition 5.3]{LX20} (see also \cite{LL16}) that $(Z, (1-\frac{r}{3}) S_1 + cS_{2,0})$ is K-semistable. Since $1-\frac{r}{3}=\frac{11+4c}{12}=b$, we know that $(Z, b S_1 + c S_2)$ is K-semistable by openness of K-semistability \cite{BLX, xu_minimizing}. By symmetry between $S_1$ and $S_2$, we know that $(Z, c S_1+b S_2)$ is also K-semistable. Since $(\frac{1}{2}, a)$ is a convex combination of $(c,b)$ and $(b, c)$, we know that $(Z, \frac{1}{2}S_1 + aS_2)$ is K-semistable.

The K-polystable part follows from similar argument to  \cite[Proof of Theorem 5.2]{adl_quartic_K3}. We include a proof here for readers' convenience. By \cite{LZ22, Zhuang} it suffices to show K-polystability of $(Z, \frac{1}{2}S_1+aS_2)$. By \cite{lwx}, there exists a special test configuration $(\cZ, \frac{1}{2}\cS_1 + a \cS_2)$ of $(Z, \frac{1}{2}S_1+aS_2)$ whose central fiber $(Z', \frac{1}{2}S_1' + a S_2')$ is K-polystable. Hence $\Fut(\cZ, \frac{1}{2}\cS_1 + a \cS_2) = 0$. By linearity of Futaki invariants and K-semistability of $(Z, bS_1+cS_2)$ and $(Z, cS_1+bS_2)$, we know that
\[
\Fut(\cZ,b\cS_1 + c \cS_2) = \Fut(\cZ,c\cS_1 +b \cS_2) = 0.
\]
By \cite[Lemma 3.1]{lwx} we know that both $(Z', bS_1' + cS_2')$ and $(Z', cS_1' + bS_2')$ are K-semistable.
Since $(S, c\Delta)$ is K-polystable, we know that $(Z, bS_1 + cS_{2,0})$ is the K-polystable degeneration of $(Z, bS_1 + cS_2)$ by \cite{LL16, LX20, projectivity_K_moduli_final}. Thus by \cite{lwx} we have a sequence of special degenerations of K-semistable log Fano pairs
\[
(Z, b S_1 + c S_2)\rightsquigarrow (Z', bS_1' + cS_2') \rightsquigarrow (Z, bS_1 + cS_{2,0}).
\]
This implies that $(Z, S_1) \cong (Z', S_1')$. By symmetry between $S_1$ and $S_2$ we also have $(Z, S_2) \cong (Z', S_2')$. Moreover, restricting the above degeneration sequence to $S_1$ yields that $(S_1, S_2|_{S_1})\cong (S_1', S_2'|_{S_1'})$. Since $S_1$ and $S_2$ are two sections of the projective cone $Z$, we conclude that $(Z, S_1+S_2)\cong (Z, S_1+ S_2)$ which implies that $(Z, \frac{1}{2}S_1+aS_2)$ is K-polystable.
\end{proof}

From Proposition \ref{prop:surface-3fold}, we see that if $c=\frac{1}{16}$ then $\frac{5+16c}{12}=\frac{1}{2}$ which precisely gives K-moduli of the double cover $X$ of $(Y, \frac{1}{2}D)$. In other words, let $\fM_{c}$ be the K-moduli space of $(S, c\Delta)$ where $S$ is a  $\bQ$-Gorenstein smoothable del Pezzo surface of degree $4$ and $\Delta \sim -4K_S$.
In the convention of \cite{wall_crossing_K_moduli}, the K-moduli space $\fM_c := KM_{\chi_0, 4, c }$ where $\chi_0$ is the Hilbert polynomial of an anti-canonically polarized smooth del Pezzo surface of degree $4$. From the following result, we see that there exists a finite injective morphism $\psi: \fM_{\frac{1}{16}}\to M_{3,4}^{\rm Kps}$ whose image is precisely the closure of the locus parametrising pure $(2,2,4)$-complete intersections $X_{2,2,4}\!\subset\!\PP(1^5, 2^2)$.
Thus to describe the compactification of the K-moduli of pure $(2,2,4)$-complete intersections in $\PP(1^5, 2^2)$, it suffices to describe the  K-moduli space $\fM_{\frac{1}{16}}$ of $(S,\frac{1}{16}\Delta)$ consisting of a del Pezzo surface $S$ of degree $4$ and $\Delta\sim -4K_S$. As a consequence, the converse of Proposition \ref{prop:surface-3fold} also holds.

\begin{thm}\label{thm:finite-morphism}
There exists a finite injective morphism $\psi: \fM_{\frac{1}{16}}\to M_{3,4}^{\rm Kps}$ whose image is the closure of the locus parameterising pure  $(2,2,4)$-complete intersections $X_{2,2,4}\!\subset\!\PP(1^5, 2^2)$.
\end{thm}

\begin{proof}
Let $\cM_c:=\cK\cM_{\chi_0, 4,c}$ be the K-moduli stack defined in \cite{wall_crossing_K_moduli} whose good moduli space is $\fM_c$. We fix $c = \frac{1}{16}$ throughout the proof. Let $\pi:(\cS, c\Delta_{\cS}) \to \cM_c$ be the universal family of K-semistable pairs. We will construct a family of K-semistable $\bQ$-Fano \threefolds{} of volume $4$ over an algebraic stack $\sW$ which is a composition of two $\bmu_2$-gerbes over $\cM_c$. This gives us a morphism of stacks $\Psi: \sW \to \cM_{3,4}^{\rm Kss}$ where the target denotes the K-moduli stack of K-semistable $\bQ$-Fano \threefolds{} of volume $4$ (see e.g. \cite{xu_survey}). Then $\psi$ is obtained by descent to good moduli spaces. The construction of such a family has three steps.

Firstly, we construct the projective cone $\cZ\to \cM_c$. Let $\cL:=\cO_{\cS}(\Delta_{\cS})$. Then we know that $\cL$ is a line bundle by Proposition \ref{prop:dP4}.  Let $\cZ:=\Proj_{\cM_c} \cR[t]$ where $\cR = \oplus_{m=0}^\infty \pi_* \cL^{\otimes m}$. Here $\pi_* \cL^{\otimes m}$ and $t$ have degree $m$ and $1$, respectively. Let $\cS_1 := \{t = 0\}$ and $\cS_2: = \{t- s = 0\}$ be two Cartier divisors in $\cZ$, where $s\in H^0(\cM_c, \pi_*\cL)$ is a section defining $\Delta_{\cS}$.

Secondly, we construct a double cover $\cY$ of $\cZ$ branched along $\cS_1$ after base change to a $\bmu_2$-gerbe. Our argument is similar to \cite[Proof of Proposition 6.12]{adl_quartic_K3}. Let $\cG:=K_{\cZ/\cM_c}+\cS_1$. Then we know that $\cG$ is a $\bQ$-Cartier Weil divisor such that $-4\cG_t \sim (\cS_1)_t$ for every $t\in |\cM_c|$. Hence $\cO_{\cZ}(4\cG + \cS_1)$ is a trivial line bundle along each fiber $\cZ_t$, which implies that it descends to a line bundle $\cF$ over $\cM_c$. Let $\phi_{\sV}:\sV\to \cM_c$ be the $\bmu_2$-gerbe obtained as the second root stack of $\cF$. Hence there exists a line bundle $\cF'_{\sV}$ on $\sV$ such that ${\cF'_{\sV}}^{\otimes 2} \cong \phi_{\sV}^* \cF$. Let $\pi_{\sV}:(\tcZ, \tcS_1+ \tcS_2)\to \sV$ be the base change of $(\cZ, \cS_1 +\cS_2)\to \cM_c$ to $\sV$. Denote by $\tcG$ the pullback of $\cG$ to $\tcZ$. Then $\cN:= \pi_{\sV}^* \cF'_{\sV} \otimes \cO_{\tcZ}(-2\tcG)$ is a $\bQ$-Cartier divisorial sheaf  such that $\cN^{[2]} \cong \cO_{\tcZ}(\tcS_1)$. We take the double cover of $\tcZ$ branched along $\tcS_1$
\[
\cY: = \Spec_{\tcZ} \cO_{\tcZ} \oplus \cN^{[-1]},
\]
where the $\cO_{\tcZ}$-algebra structure is induced by $\cN^{[-2]}\xrightarrow{\cdot\tilde{s}} \cO_{\tcZ}$ where $\tilde{s}$ is a section of $\cN^{[2]}$ such that $\{\tilde{s}=0 \}= \tcS_1$. Let $\cD$ be the pullback of $\tcS_2$ to $\cY$.

Lastly, we construct a double cover $\mathcal{X}$ of $\cY$ branched along $\cD$ after base change to a $\bmu_2$-gerbe. The construction is similar to the previous step, so we omit the details. In the end, we obtain a $\bmu_2$-gerbe $\sW \to \sV$ such that we can take a double cover $\mathcal{X}$ of $\tcY$ branched along $\tcD$ where $(\tcY, \tcD) = (\cY, \cD) \times_{\sV} \sW$. By Kawamata-Viehweg vanishing, the above construction commutes with base change and hence the family $\mathcal{X} \to \sW$ has K-semistable fibers by Proposition \ref{prop:surface-3fold}  where a general fiber is a pure $(2,2,4)$-complete intersection in $\bP(1^5, 2^2)$. Thus we obtain a morphism $\Psi: \sW \to \cM_{3,4}^{\rm Kss}$. Moreover, $\sW$ admits a good moduli space isomorphic to $\fM_c$ by \cite{alper} (see also \cite[Proof of Proposition 6.12]{adl_quartic_K3}). Hence $\Psi$ descends to a morphism between good moduli spaces $\psi: \fM_{\frac{1}{16}} \to M_{3,4}^{\rm Kps}$ by \cite[Theorem 6.6]{alper}. It remains to show that $\psi$ is a finite injective morphism. Since $\fM_{\frac{1}{16}}$ is proper, it suffices to show that $\psi$ is injective on $\bC$-points.

Suppose $\psi(S, \Delta)\cong \psi(S', \Delta')$. To show $\psi$ is injective, it suffices to show $(S,\Delta)\cong (S', \Delta')$. By Proposition \ref{prop:dP4}, we know that both $S$ and $S'$ are del Pezzo surfaces of degree $4$ with du Val singularities. Thus $S$ and $S'$ are both $(2,2)$-complete intersections in $\bP^4$, and we may write
\[
(S,\Delta) = (\{g= h= 0\}, \{f=0\}|_S) \quad\textrm{and} \quad (S',\Delta') = (\{g'=h'=0\}, \{f'=0\}|_{S'}),
\]
where $g,h,g',h'$ (resp. $f,f'$) are homogeneous polynomials of degree $2$ (resp. of degree $4$) in $(x_0, x_1, x_2, x_3, x_4)$. Let $X = \psi(S,\Delta)$ and $X'=\psi(S',\Delta')$. Then we have $X\cong X'$ by assumptions. Moreover, we  know that $X = \{g=h = y_0y_1 -f=0\}$ and $X' =  \{g'=h'=y_0y_1 -f'=0\}$ are pure $(2,2,4)$-complete intersections in $\bP(1^5, 2^2)$. Since the embedding $X\hookrightarrow \bP(1^5,2^2)$ is induced by the  linear systems $|-K_X|$ and $|-2K_X|$: the coordinates $(x_0, \cdots, x_4)$ form a basis of $H^0(X,-K_X)$, and the coordinates $(y_0, y_1)$ are lifting of a basis of $\mathrm{coker}(\mathrm{Sym}^2 H^0(X,-K_X) \to H^0(X, -2K_X))$, we know that $X$ and $X'$ are projectively equivalent as weighted complete intersections. Thus under a change of coordinates $\sigma\in \Aut(\bP(1^5, 2^2))$ we have $X' = \sigma^* X$. We may decompose $\sigma = \sigma_x \circ \sigma_y$ where $\sigma_x$ is linear in $x_i$ and fixes $y_0$ and $y_1$, and $\sigma_y$  fixes each $x_i$. Then a simple analysis of the transformation of equations under $\sigma^*$ shows that $(S',\Delta')=\sigma_x^* (S,\Delta)$ where we treat $\sigma_x$ as an element in $\mathrm{PGL}_5(\bC)$. Thus the proof is finished.
\end{proof}

The next result bounds the singularities of surfaces appearing in $\fM_c$.

\begin{prop}\label{prop:dP4}
Let $[(S,c\Delta)]\in \fM_{c}$ be a K-polystable pair for some $c\in (0, \frac{1}{16}]$. Then either $S$ is smooth, or $S$ has only $A_1$, $A_2$, or $A_3$-singularities. If in addition $S$ has an $A_3$-singularity, then we have
\[
c=\frac{1}{16},~ S \cong \{y_3^2 = y_2 y_4\}\subset \bP(1,2,3,4)_{[y_1,y_2,y_3,y_4]}, ~\textrm{and} ~\Delta = \{y_4^4 =0\}.
\]
Note that this surface $S$ is \textnumero 25 from \cite[Big Table]{CP20} (see also Table \ref{table:dP4-Gm}).
\end{prop}

\begin{proof}
We know that
\[
(-K_S-c\Delta)^2 = (1-4c)^2 (-K_S)^2 = 4(1-4c)^2.
\]
Let $x\in S$ be a singular point with local orbifold group $G_x$.
The local-global volume comparison \cite{LL16} implies
\begin{equation}\label{eq:vol-compare}
\frac{4}{|G_x|}=\hvol(x, S) \geq \hvol(x, S, c\Delta) \geq  \frac{4}{9}(-K_S-c\Delta)^2 =\frac{16}{9}(1-4c)^2 \geq 1.
\end{equation}
Thus we have that $|G_x|\leq 4$. This, together with the $T$-singularity condition (see \cite{kollar_shepherd_barron}), implies that $x\in S$ is $A_1$, $A_2$, $A_3$ or of type $\frac{1}{4}(1,1)$.

Next, we rule out the case of type $\frac{1}{4}(1,1)$. Assume to the contrary that $x\in S$ has type $\frac{1}{4}(1,1)$. Since $|G_x|=4$, every inequality from \eqref{eq:vol-compare} is an equality. Thus we have $x\not\in \Supp(\Delta)$, $c=\frac{1}{16}$, and the inequality in the local-global volume comparison from \cite{Liu18, LL16} is an equality.  By \cite[Remark 38(1)]{Liu18}, we have that $S\cong \bP^2/\bmu_4$ where the $\bmu_4$-action on $\bP^2$ has weight $(1,1,0)$, and $\frac{1}{16}\Delta$ corresponds to the orbifold divisor of $\bP^2/\bmu_4$. In particular, we have $S\cong \bP(1,1,4)$ which implies that $(K_S^2) = 9\neq 4$, a contradiction. Thus singularities of type $\frac{1}{4}(1,1)$ cannot appear.

Finally, we consider the case where $x\in S$ has type $A_3$. Similar to the previous case of type $\frac{1}{4}(1,1)$, we have $x\not\in \Supp(\Delta)$, $c=\frac{1}{16}$. Moreover, $S\cong \bP^2/\bmu_4$ where the $\bmu_4$-action on $\bP^2$ has weight $(1,3,0)$, and $\frac{1}{16}\Delta$ corresponds to the orbifold divisor of $\bP^2/\bmu_4$.  Suppose a generator $g$ of $\bmu_4$ acting on $\bP^2_{[x,y,z]}$ as
\[
g\cdot [x,y,z] = [ix, -i y, z]  = [-x, y, iz].
\]
Denote by $R=\oplus_{m=0}^\infty H^0(\bP^2, \cO(m))$ the graded section ring of $(\bP^2, \cO(1))$.
If we pick the latter $\bmu_4$-linearization on $(\bP^2, \cO(1))$ of weight $(2,0,1)$, then the $\bmu_4$-invariant subring $R^{\bmu_4}$ is generated by $y,x^2,xz^2, z^4$. Thus we have
\[
\bP^2/\bmu_4 \cong \Proj R^{\bmu_4} \cong \{(y_3^2 = y_2y_4\}\subset \bP(1,2,3,4)),
\]
where the generators corresponds to $y_1 = y$, $y_2 = x^2$, $y_3=xz^2$ and $y_4 = z^4$. Since the $\bmu_4$-action has one orbifold curve given by $(z=0)$ of order $2$ whose reduced image in $\bP^2/\bmu_4$ corresponds $\frac{1}{2}(y_4=0)$, we have $\frac{1}{16}\Delta = \frac{1}{2}\cdot \frac{1}{2}(y_4=0)$ which implies that $\Delta = \{y_4^4 =0\}$. It is easy to check that $(K_S^2) = 4$ and $\Delta \sim -4K_S$, hence $[(S, \frac{1}{16}\Delta)]\in \fM_{\frac{1}{16}}$. The proof is finished.
%, then the equality from  is satisfied, which implies that $(S, \frac{1}{16}\Delta) \cong \bP^2/(\bZ/4\bZ)$ where the action has weight $(1,3,0)$ (respectively $(1,1,0)$) by \cite[Remark 38(1)]{Liu18}. But then we get $S\cong (x_0 x_1=x_2^4)\subset \bP(4,4,2,1)$ (respectively $S\cong \bP(1,1,4)$) which is a contradiction as they have different degrees. Thus neither type $A_3$ nor $\frac{1}{4}(1,1)$ can appear, and the proof is finished.
\end{proof}

\begin{cor}\label{cor:moduli}
The K-moduli compactification of pure $(2,2,4)$-complete intersections $\PP(1^5,2^2)$ only consists  of complete intersections of the same type.
\end{cor}

\begin{proof}
By Proposition \ref{prop:dP4}, we see that every pair $(S,\Delta)$ in $\fM_{\frac{1}{16}}$ satisfies that $S$ has du Val singularities. Hence $S=\{g=h=0\}$ is always a complete intersection of two quadrics in $\bP^4$ and $\Delta = \{f=0\}|_S$, which implies that $Y=C_p(S, -2K_S)=\{g=h=0\}$ is the complete intersection in $\PP(1^5,2)$. Hence $\psi(S,\Delta)$ is a complete intersection $X_{2,2,4}\!\subset\!\PP(1^5,2^2)$ given by $\{y_0y_1 = f\}$ and $\{g=h=0\}$. Thus the statement follows from Theorem \ref{thm:finite-morphism}.
\end{proof}

\section{Wall crossing}\label{wall}

In this section, we study the explicit wall crossings for the K-moduli spaces $\fM_c$ where $c\in (0, \frac{1}{16}]$. Similar problems were studied in \cite{wall_crossing_K_moduli, K-moduli_quadric_K3, adl_quartic_K3, theo, Zhao22}.

Let us recall the convention from \cite{wall_crossing_K_moduli}. A wall $c=c_i$ of the K-moduli spaces $\fM_c$ is a rational value such that K-polystability changes from $c=c_i-\epsilon$ to $c=c_i+\epsilon$ for $0<\epsilon \ll 1$. By \cite[Theorems 1.1 and 1.2]{wall_crossing_K_moduli}, there are finitely many walls $0< c_1<c_2<\cdots< c_k\leq\frac{1}{16}$ for the K-moduli spaces $\fM_c$ with $c\in (0,\frac{1}{16}]$. Moreover, we have wall crossing diagrams of birational morphisms
\[
 \fM_{c_i-\epsilon}\to \fM_{c_i} \leftarrow \fM_{c_i+\epsilon}.
\]
A K-polystable pair $(S,c_i\Delta)$ in $\fM_{c_i}$ is called a \emph{new K-polystable pair on the wall $c_i$} if $(S,c\Delta)$ is K-unstable whenever $c\neq c_i$. By \cite{wall_crossing_K_moduli} we know that every new K-polystable pair on a wall  admits an effective $\bG_m$-action, hence is not K-stable; on the other hand, every wall admits a new K-polystable pair.

In addition to walls in $(0,\frac{1}{16}]$, there is an initial wall crossing at $c_0=0$, where we will have a fibre-type contraction $\fM_{\epsilon}\to \fM_0$. Here $\fM_0$ is the K-moduli space of $\bQ$-Gorenstein smoothable del Pezzo surfaces of degree $4$, which is literally the K-moduli space of $(S,0\Delta)$. By \cite{mabuchi_mukai, odaka_spotti_sun} we know that  $\fM_0$ is isomorphic to the GIT moduli space of complete intersections of two quadric hypersurfaces in $\bP^4$. Note that this initial wall crossing was not explicitly stated in \cite{wall_crossing_K_moduli}, although the techniques therein are sufficient to deduce such a result.

\begin{thm}\label{thm:full-wall-crossing}
    There are four walls for $c\in (0,\frac{1}{16}]$ for the K-moduli spaces $\fM_c$: $c_1=\frac{1}{28}$, $c_2=\frac{1}{22}$, $c_3=\frac{1}{19}$, $c_4 = \frac{1}{16}$. The new K-polystable pairs $(S,c_i\Delta)$ on the walls are completely classified  in Proposition \ref{prop:dP4}, Theorems \ref{thm:walls-29} and \ref{thm:walls-33} where $S$ is isomorphic to \textnumero 25, 29, and 33 in Table \ref{table:dP4-Gm} respectively. Moreover, the initial wall crossing at $c_0=0$ gives a surjective morphism $\fM_{\epsilon}\to \fM_0$ with connected fibres whose general fibre is isomorphic to $\bP^{40}/\bmu_2^4$.
\end{thm}

To find all walls for the K-moduli spaces $\fM_c$ where $c\in (0,\frac{1}{16}]$, we need to classify new K-polystable pairs $(S, c\Delta)$, i.e.\ it is K-polystable but not K-stable, and $S$ is  K-unstable. This way we can obtain a complete list of walls and new K-polystable pairs on the wall.

 By Proposition \ref{prop:dP4}, we know that either $S$ is a K-unstable del Pezzo surface of degree $4$ with at worst $A_2$ singularities that admits a $\bG_m$-action, or $S$ is isomorphic to  \textnumero 25.
The following table from \cite[Big Table]{CP20} classifies del Pezzo surface of degree $4$ with at worst $A_2$ singularities that admits a $\bG_m$-action, with an additional row of \textnumero 25 which shows up in $\fM_{\frac{1}{16}}$ only.

\begin{center}
\begin{table}[htbp!]\renewcommand{\arraystretch}{1.5}
\caption{Del Pezzo surfaces of degree $4$ with $\bG_m$-actions and at worst $A_2$-singularities, together with \textnumero 25}\label{table:dP4-Gm}
\begin{tabular}{|c|c|c|c|c|}
\hline
\textnumero & $\rho$ & \textbf{Sing.}   & $\Aut^0$ & \textbf{Equation \& Ambient space}\\ \hline \hline
29 & $2$ & $A_22A_1$  & $\bG_m^2$ & $\{y_2y_2' = y_1^3 y_1'\}\subset \bP(1,1,2,2)$   \\ \hline
30 & $2$ & $4A_1$  & $\bG_m^2$  & $\{y_2y_2' = y_1^2 y_1'^2\}\subset \bP(1,1,2,2)$
        \\ \hline
33 & $3$ & $A_2  A_1$  & $\bG_m$  & $\{x_0 x_1 -x_2 x_3 = x_1x_2 + x_2x_4 + x_3x_4=0\}\subset \bP^4$
        \\ \hline
34 & $3$ & $3A_1$  & $\bG_m$  & $\{y_2y_2' = y_1^2 y_1' (y_1 + y_1')\}\subset \bP(1,1,2,2)$
        \\ \hline
35 & $4$ & $2A_1$  & $\bG_m$  & $\{y_2y_2' = y_1 y_1' (y_1' - y_1)(y_1' -\lambda y_1)\}\subset \bP(1,1,2,2)$ for $\lambda\in \bC\setminus\{0,1\}$
        \\ \hline\hline
25 & $1$ & $A_3 2A_1$ & $\mathbb{B}_2 \times \bG_m$ & $\{y_3^2 = y_2 y_4\}\subset \bP(1,2,3,4)$      \\
\hline
\end{tabular}
\end{table}
\end{center}

Here $\mathbb{B}_2$ represents the Borel subgroup of $\mathrm{PGL}_2(\mathbb{C})$.

From Table \ref{table:dP4-Gm} and \cite{mabuchi_mukai, odaka_spotti_sun}, we see that \textnumero 25, 29, 33 are K-unstable, while \textnumero 30, 35 are K-polystable, and \textnumero 34 is K-semistable but not K-polystable.

From now on, we study the walls from $\fM_\epsilon$ to $\fM_{\frac{1}{16}}$ which involve only \textnumero 25, 29 or 33. Since there is only one wall
at $c=\frac{1}{16}$ involving \textnumero 25 by Proposition \ref{prop:dP4}, we will focus on walls involving \textnumero 29 or 33.

\subsection{\textnumero 29}
Here $S= \{y_2 y_2' = y_1^3 y_1'\}\subset \bP(1,1,2,2)_{[y_1,y_1',y_2,y_2']}$.

\begin{thm}\label{thm:walls-29}
There are precisely four walls in $(0,\frac{1}{16}]$ involving $S$ from \textnumero 29 such that $(S,c_i\Delta_i)$ is a new K-polystable pair for the following choices of $c_i$ and $\Delta_i$.
\begin{enumerate}
    \item $c_1=\frac{1}{28}$, $\Delta_1= \{y_1'^8 =0\}$.
    \item $c_2= \frac{1}{22}$, $\Delta_2 = \{y_1 y_1'^7=0\}$.
    \item $c_3= \frac{1}{19}$, $\Delta_3 = \{(y_2+y_2')y_1'^6=0\}$.
    \item $c_4= \frac{1}{16}$, $\Delta_4 = \Delta_{4,\mathbf{a}}$ or $\Delta_{4,\infty}$, where
    \[
    \Delta_{4,\mathbf{a}} = \{y_2 y_1'^6 + a_1 y_1^2 y_1'^6 + a_2 y_1 y_2' y_1'^5 + a_3 y_2'^2 y_1'^4=0\}~\textrm{ for }~\mathbf{a} = (a_1,a_2,a_3)\in \bC\times (\bC^2\setminus\{(0,0)\}),
    \]
    and $\Delta_{4,\infty}=\{y_1^2 y_1'^6=0\}$.

%    \YL{move elsewhere}$(a_0 y_2 y_1'^6 + a_1 y_1^2 y_1'^6 + a_2 y_2'^2 y_1'^4)=0)$, where either $a_0=a_2=0$, $a_1\neq 0$ or $a_0\neq 0$, $a_2\neq 0$.
\end{enumerate}
\end{thm}

We first show that these pairs are K-polystable. The fact that $S$ is toric can be seen by
\[
(s,t)\cdot [y_1,y_1',y_2, y_2'] = [sy_1,s^{-3}y_1',ty_2, t^{-1}y_2']
\]
It is clear that $\Delta_i$ for $i\leq 2$ is toric. Thus it suffices to show the Futaki character $\Fut_{(S, c_i \Delta_i)}=0$. In addition, $S$ admits an involution $\tau: S\to S$ given by $\tau([y_1,y_1',y_2, y_2'])=[y_1,y_1',y_2', y_2]$.

Denote by $C=\{y_1'=y_2=0\}$, $C'=\{y_1'=y_2'=0\}$, $B= \{y_1=y_2 =0\}$, $B'= \{y_1=y_2'=0\}$.
The intersection numbers of these curves are summarized in Table \ref{table:No29}.
\begin{table}[ht!]
\caption{\label{table:No29} Intersection numbers of the surface \textnumero 29}
\begin{center}
\renewcommand\arraystretch{1.2}
\begin{tabular}{|c||c|c|c|c|}
\hline
$\bullet$  & $C$ & $C^\prime$ & $B$ & $B'$ \\
\hline\hline
$C$        & $-\frac{1}{2}$ & $1$ & $\frac{1}{2}$ & $0$\\
\hline
$C^\prime$  & $1$ & $-\frac{1}{2}$ & $0$ & $\frac{1}{2}$\\
\hline
$B$ & $\frac{1}{2}$ & $0$ & $\frac{1}{6}$ & $\frac{1}{3}$\\
\hline
$B'$ & $0$ & $\frac{1}{2}$ & $\frac{1}{3}$ & $\frac{1}{6}$\\
\hline
\end{tabular}
\end{center}

\end{table}

\begin{lem}\label{lem:NE29}
The Mori cone $\overline{NE}(S)$ is generated by $[C]$ and $[C']$.
\end{lem}

\begin{proof}
    From the torus action on $S$, we see that there are precisely four torus invariant divisors: $C$, $C'$, $B$, and $B'$. Thus $\overline{NE}(S)$ is generated by these four curves.  From the intersection table above we see that both $B$ and $B'$ are nef and big, while both $C$ and $C'$ are extremal. Thus the proof is finished.
\end{proof}

We first compute $S$-invariant of these curves.

\begin{prop}
    Under the above notation, we have
    \[
    S_S(C) = S_S(C') = \frac{5}{6}, \qquad S_S(B) = S_S(B') = \frac{7}{6}.
    \]
\end{prop}

\begin{proof}
First of all, the involution $\tau$ satisfies $\tau(C) = C'$ and $\tau(B) = B'$. Thus we have equalities $S_S(C) = S_S(C')$ and $S_S(B) = S_S(B')$.

Next, we compute $S_S(C)$. It is clear that $-K_S = \cO_S(2) = 2(C+C')$. This gives
\[
S_S(C) = \frac{1}{4}\int_0^\infty \vol_S((2-t)C + 2C') dt.
\]
It follows from Table \ref{table:No29} and Lemma \ref{lem:NE29} that $(2-t)C + 2C'$ is nef if  $0\leq t\leq 1$, and not big if $t\geq 2$. For $1\leq t\leq 2$, the Zariski decomposition is
\[
 (2-t) C + 2C' = P+N =  (2-t) (C+ 2C') + (2t-2) C'.
\]
Thus computation shows
\[
\vol((2-t)C + 2C') = \begin{cases}
4-2t - \frac{t^2}{2} & \textrm{ if }0\leq t\leq 1;\\
\frac{3}{2}(2-t)^2 & \textrm{ if }1\leq t\leq 2.
\end{cases}
\]
Thus
\[
S_S(C) = \frac{1}{4}\left( \int_0^1 (4-2t - \frac{t^2}{2})dt + \int_1^2 \frac{3}{2}(2-t)^2 dt\right) = \frac{5}{6}.
\]

Finally, we compute $S_S(B)$. It is clear that $-K_S = \cO_S(2) = 2(B+B')$. This gives
\[
S_S(B) = \frac{1}{4}\int_0^\infty \vol_S((2-t)B + 2B') dt.
\]
Moreover, by Table \ref{table:No29} and Lemma \ref{lem:NE29}, we know that  $(2-t)B+ 2B'$ is nef if $0\leq t\leq 2$, and not big if $t\geq 3$. For $2\leq t\leq 3$, the Zariski decomposition is
\[
 (2-t) C + 2C' = P+N =  (6-2t) B' + (t-2) C.
\]
Thus computation shows
\[
\vol((2-t)B + 2B') = \begin{cases}
4-2t + \frac{t^2}{6} & \textrm{ if }0\leq t\leq 2;\\
\frac{2}{3}(3-t)^2 & \textrm{ if }2\leq t\leq 3.
\end{cases}
\]
Thus
\[
S_S(B) = \frac{1}{4}\left( \int_0^2 (4-2t + \frac{t^2}{6})dt + \int_2^3 \frac{2}{3}(3-t)^2 dt\right) = \frac{7}{6}.
\]
\end{proof}

\begin{prop}\label{prop:29-1}
    The pairs $(S,c_i\Delta_i)$ from Theorem \ref{thm:walls-29} are K-polystable and not K-stable.
\end{prop}
\begin{proof}
We first look at cases (1) and (2). For these two cases, $(S, c_i \Delta_i)$ is toric hence not K-stable. Thus we just need to show that the Futaki character $\Fut_{(S, c_i\Delta_i)}=0$. Since the $1$-PS induced by $\ord_C$, $\ord_{C'}$, $\ord_B$, and $\ord_B'$ generates the torus, it suffices to show that $\beta_{(S, c_i\Delta_i)}(v)=0$ where $v$ is one of the four valuations. By symmetry of $\tau$, it suffices to show that $\beta_{(S, c_i\Delta_i)}(C)=\beta_{(S, c_i\Delta_i)}(B)=0$.

Since $\Delta_i \equiv -4K_S$, we know that
$A_{(S, c_i\Delta_i)}(v) = A_S(v)- c_i v(\Delta_i)$ and $S_{(S, c_i\Delta_i)}(v) = (1-4c_i) S_S(v)$. Thus
\[
\beta_{(S, c_i\Delta_i)}(v)=A_S(v)- c_i v(\Delta_i) - (1-4c_i) S_S(v)
\]

(1) Since $c_1=\frac{1}{28}$ and $\Delta_1 = (y_1'^8 =0)$, we have
\begin{align*}
\beta_{(S, c_1\Delta_1)}(C) & = 1-\frac{8}{28} - (1-\frac{4}{28})  \frac{5}{6} = 0,\\
\beta_{(S, c_1\Delta_1)}(B) & = 1 - (1-\frac{4}{28})  \frac{7}{6} = 0.
\end{align*}

(2) Since $c_2=\frac{1}{22}$ and $\Delta_2 = (y_1 y_1'^7 =0)$, we have
\begin{align*}
\beta_{(S, c_2\Delta_2)}(C) & = 1-\frac{7}{22} - (1-\frac{4}{22})  \frac{5}{6} = 0,\\
\beta_{(S, c_2\Delta_2)}(B) & = 1 - \frac{1}{22} - (1-\frac{4}{22})  \frac{7}{6} = 0.
\end{align*}

Next, we look at case (3). In this case, $(S, \Delta_3)$ admits a $\bG_m$-action $\sigma$ given by
\[
s\cdot [y_1, y_1', y_2, y_2']= [sy_1, s^{-3}y_1', y_2, y_2'].
\]
Thus $(S, c_3\Delta_3)$ is not K-stable.
In the affine chart $U:=\{y_1=1\}\subset S$, we have an isomorphism $U\cong \bA^2$ given by $[1,y_1',y_2,y_2']\mapsto (y_2, y_2')$ where $C|_U$ and $C'|_U$ corresponds to $(y_2=0)$ and $(y_2'=0)$ respectively.
If we denote by $v_0= \ord_{[1,0,0,0]}$, then  $v_0$ is the quasi-monomial combination of $\ord_C$ and $\ord_{C'}$ of weight $(1,1)$ (see e.g. \cite{JM12} for definition). Thus we have
\[
\beta_S(v_0) = \beta_S(C) + \beta_S(C') = 2\beta_S(C) = 2(1-S_S(C)) = \frac{1}{3}.
\]
It is clear that $v_0$ induces the $\bG_m$-action $\sigma$. Thus by \cite[Theorem 3.2]{Liu23} (cf.\ \cite[Theorem 1.3.9]{calabi_problem_3folds} and \cite{IS17}) it suffices to show that $\beta_{(S,c_3\Delta_3)}(v_0) = 0$, and that $\beta_{(S,c_3\Delta_3)}(D)>0$ for every vertical prime divisor $D$ on $S$.

We first compute $\beta$-invariant of $v_0$. Since $\beta_S(v_0) = \frac{1}{3}$ and $A_S(v_0)=2$, we know that $S_S(v_0) = \frac{5}{3}$. Thus
\[
\beta_{(S,c_3\Delta_3)}(v_0) = 2-c_3v_0(\Delta_3) - (1-4c_3) S_S(v_0) = 2- \frac{13}{19} - (1-\frac{4}{19})\frac{5}{3} =0.
\]

Next we compute $\beta$-invariant of vertical divisors $D$. A simple analysis of the $\sigma$-action on $S$ shows that $D$ is one of the following:
\begin{enumerate}[(i)]
    \item $C$ or $C'$;
    \item $B$ or $B'$;
    \item $D_a = \{y_2 + ay_2'=0\}$ for $a\neq 0$.
\end{enumerate}

Next, we split into these three cases.

(i) Since $\Delta_3$ is invariant under the involution $\tau$, we have
\[
\beta_{(S,c_3\Delta_3)}(C) = \beta_{(S,c_3\Delta_3)}(C') = 1 - \frac{6}{19} - (1-\frac{4}{19})\frac{5}{6} = \frac{1}{38}>0.
\]

(ii) Similarly, we have
\[
\beta_{(S,c_3\Delta_3)}(B) = \beta_{(S,c_3\Delta_3)}(B') = 1  - (1-\frac{4}{19})\frac{7}{6} = \frac{3}{38}>0.
\]

(iii) We know that $\ord_{D_a}(\Delta_3) \leq 1$. Moreover, we know that $-K_S - tD_a = \cO_S(2 - 2t)$ which is nef if $0\leq t\leq 1$ and not big if $t\geq 1$. Hence
\[
S_S(D_a) = \frac{1}{4}\int_0^1 (2-2t)^2 dt = \frac{1}{3}.
\]
Thus
\[
\beta_{(S,c_3\Delta_3)}(D_a) \geq 1  -\frac{1}{19} - (1-\frac{4}{19})\frac{1}{3} = \frac{13}{19}>0.
\]

Finally, we look at  case (4). We first treat the case of $\Delta_{4,\infty}$. %If $a_0=a_2=0$ and $a_1\neq 0$,
It is clear that $\Delta_{4,\infty} = \{y_1^2 y_1'^6=0\}$ is toric. Hence $(S, c_4\Delta_{4,\infty})$ is not K-stable. Since $c_4=\frac{1}{16}$, we have
\begin{align*}
\beta_{(S, c_4\Delta_{4,\infty})}(C) & = 1-\frac{6}{16} - (1-\frac{4}{16})  \frac{5}{6} = 0,\\
\beta_{(S, c_4\Delta_{4,\infty})}(B) & = 1 - \frac{2}{16} - (1-\frac{4}{16})  \frac{7}{6} = 0.
\end{align*}
This shows that $(S, c_4\Delta_{4,\infty})$ is K-polystable.

Next, we treat the case of $\Delta_{4,\mathbf{a}}$ for $\mathbf{a}= (a_1,a_2,a_3)\in \bC\times (\bC^2\setminus \{(0,0)\})$.
%If $a_0\neq 0$ and $a_2\neq 0$, then after rescaling we can assume that $a_0=a_2=1$ and $a_1=a\in \bC$, i.e. $\Delta_4 = (y_2y_1'^6 + a y_1^2 y_1'^6 + y_2'^2 y_1'^4=0)$.
Since $\Delta_{4,\mathbf{a}} = \{y_2 y_1'^6 + a_1 y_1^2 y_1'^6 + a_2 y_1 y_2' y_1'^5 + a_3 y_2'^2 y_1'^4=0\}$, we know that   $(S, \Delta_{4,\mathbf{a}})$ admits a $\bG_m$-action $\sigma'$ given by
\[
t\cdot [y_1, y_1', y_2, y_2']= [ty_1, y_1', t^2 y_2, t y_2'].
\]
Thus $(S, c_4\Delta_{4,\mathbf{a}})$ is not K-stable.
In the affine chart $U':= \{y_1'=1\}\subset S$, we have an isomorphism $U'\cong \{y_1^3 = y_2 y_2'\}\subset \bA^3_{(y_1, y_2, y_2')}$. Let $v_1$ be the monomial valuation on $U'$ of weight $(1, 2, 1)$ in $(y_1,y_2,y_2')$ centered at $[0,1,0,0]$. Then $v_1$ is a quasi-monomial combination of $\ord_B$ and $\ord_{B'}$ of weight $(\frac{2}{3},\frac{1}{3})$. Thus we have
\[
\beta_S(v_1) =  \frac{2}{3}\beta_S(B) + \frac{1}{3}\beta_S(B') = \beta_S(B) = -\frac{1}{6}.
\]
It is clear that $v_1$ induces the $\bG_m$-action $\sigma'$. Thus by \cite[Theorem 3.2]{Liu23} (cf.\ \cite[Theorem 1.3.9]{calabi_problem_3folds} and \cite{IS17}) it suffices to show that $\beta_{(S,c_4\Delta_{4,\mathbf{a}})}(v_1) = 0$, and that $\beta_{(S,c_4\Delta_{4,\mathbf{a}})}(D)>0$ for every vertical prime divisor $D$ on $S$.

We first compute $\beta$-invariant of $v_1$.
Since $A_S(v_1) = \frac{2}{3} A_S(B) + \frac{1}{3}A_S(B') = 1$, we know that $S_S(v_1) = A_S(v_1) - \beta_S(v_1) = \frac{7}{6}$. Moreover, we have $\Delta_{4,\mathbf{a}}|_{U'} = \{y_2 + a_1 y_1^2 + a_2 y_1y_2' + a_3 y_2'^2=0\}$ which implies $v_1(\Delta_{4,\mathbf{a}}) = 2$. Thus
\[
\beta_{(S,c_4\Delta_{4,\mathbf{a}})}(v_1) = 1-c_4v_1(\Delta_{4,\mathbf{a}}) - (1-4c_4) S_S(v_1) = 1- \frac{2}{16} - (1-\frac{4}{16})\frac{7}{6} =0.
\]
We also observe that $C'$ is the other horizontal divisor of $\sigma'$, whose $\beta$-invariant is automatically zero as $\beta=\Fut$ is linear on the space of cocharacters.

Next we compute $\beta$-invariant of vertical divisors $D$. A simple analysis of the $\sigma'$-action on $S$ shows that $D$ is one of the following:
\begin{enumerate}[(i)]
    \item $C$;
    \item $B$ or $B'$;
    \item $D_b = (y_2 + by_1^2=0) - 2B$ for $b\neq 0$.
\end{enumerate}

Next, we split into these three cases.

(i) Since $\ord_C(\Delta_{4,\mathbf{a}}) = 4$ or $5$, we have
\[
\beta_{(S,c_4\Delta_{4,\mathbf{a}})}(C)  \geq  1 -\frac{5}{16} - (1-\frac{4}{16})\frac{5}{6} = \frac{1}{16}>0.
\]

(ii) Since $\ord_{B'}(\Delta_{4,\mathbf{a}})=0$ and $\ord_{B}(\Delta_{4,\mathbf{a}})=0$ or $1$, we have
\[
\beta_{(S,c_4\Delta_{4,\mathbf{a}})}(B') \geq \beta_{(S,c_4\Delta_{4,\mathbf{a}})}(B) \geq 1 -\frac{1}{16}  - (1-\frac{4}{16})\frac{7}{6} = \frac{1}{16}>0.
\]

(iii) In the affine chart $U=\{y_1=1\}\cong \bA^2_{(y_2,y_2')}$, we have $\Delta_{4,\mathbf{a}}|_U = \{y_2^4 y_2'^6(y_2^3 + a_1y_2^2 + a_2y_2 + a_3)=0\}$ and $D_b|_U= \{y_2 + b =0\}$. Hence $\ord_{D_b}(\Delta_{4,\mathbf{a}}) \leq 3$ as $b\neq 0$. Moreover, we know that $D_b \equiv \cO_S(2) - 2B\equiv 2B'$. Hence
\[
S_S(D_b) = \frac{1}{4}\int_0^{\infty} \vol_S(-K_S - 2tB') dt = \frac{1}{2}S_S(B') = \frac{7}{12}.
\]
Thus
\[
\beta_{(S,c_4\Delta_{4,\mathbf{a}})}(D_b) \geq 1  -\frac{3}{16} - (1-\frac{4}{16})\frac{7}{12} = \frac{3}{8}>0.
\]
\end{proof}

\begin{prop}\label{prop:29-2}
Let $S$ be the surface \textnumero 29 from Table \ref{table:dP4-Gm}. Suppose $\Delta \in |-4K_S|$ and $c\in (0,\frac{1}{16}]$ satisfy that $(S, c\Delta)$ is K-polystable but not K-stable. Then $(S, c\Delta)\cong (S, c_i\Delta_i)$ for some $1\leq i\leq 4$ and $c_i$, $\Delta_i$ described in Theorem \ref{thm:walls-29}.
\end{prop}

\begin{proof}
%Assume that $(S, c\Delta)$ is K-polystable but not K-stable where $\Delta \in |-4K_S|$ and $c\in (0,\frac{1}{16}]$.
Let $f(y_1,y_1', y_2, y_2')$ be the defining polynomial of $\Delta$. Since $S=\{y_2y_2' = y_1^3 y_1'\}\subset\bP(1,1,2,2)$, we may assume that no monomial in $f$ is  divisible by $y_2y_2'$. We split into the following cases.

(a) $y_1'^8$ appears in $f$.
Then $\Delta$ admits a special degeneration to $\{y_1'^8 = 0\} = \Delta_1$ under the $\bG_m$-action $\sigma$. Thus $(S, c_1\Delta)$ is K-semistable whose K-polystable degeneration  is $(S, c_1\Delta_1)$. This shows that $(S, \Delta)\cong (S, \Delta_1)$ and $c= c_1$ by \cite[Proposition 3.18]{wall_crossing_K_moduli}.

(b) $y_1'^8$ does not appear in $f$, while $y_1 y_1'^7$ appears in $f$. Then $\Delta$ admits a special degeneration to $\{y_1 y_1'^7 = 0\} = \Delta_2$ under the $\bG_m$-action $\sigma$. By similar argument as in case (a), we conclude that $(S, \Delta)\cong (S, \Delta_2)$ and $c=c_2$.

(c) Both $y_1'^8$ and  $y_1 y_1'^7$ do not appear in $f$, while $y_1'^6$ divides some monomial in $f$. Let $(a_1y_1^2 + a_0 y_2 + a_0' y_2') y_1'^6$ be the sum of all terms in $f$ divisible by $y_1'^6$.

If $a_0\neq 0$ and $a_0'\neq 0$, then $\Delta$ admits a special degeneration to $\Delta' := \{(a_0 y_2 + a_0' y_2') y_1'^6 = 0\}$ under the $\bG_m$-action $\sigma$. After rescaling we have $(S, \Delta')\cong (S, \Delta_3)$. By similar argument as in case (a), we conclude that $(S, \Delta)\cong (S, \Delta_3)$ and $c=c_3$.

If either $a_0=0$ or $a_0'=0$ but not both, we will show that $c=c_4$ and $(S,\Delta) \cong (S,\Delta_{4,\mathbf{a}})$ for some $\mathbf{a}\in\bC\times(\bC^2\setminus\{(0,0)\})$. By symmetry we may assume that $a_0\neq 0$ and $a_0'=0$. Hence after rescaling we may assume that $a_0 = 1$.
Under the $\bG_m$-action $\sigma'$ induced by the valuation $v_1$, then we see that $v_1(\Delta) = 2$ and $\Delta$ admits a special degeneration to $\Delta' = \{y_2 y_1'^6 + a_1 y_1^2 y_1'^6 + a_2 y_1 y_2' y_1'^5 + a_3 y_2'^2 y_1'^4=0\}$. Moreover, we compute
\[
\beta_{(S,c\Delta)}(v_1) = 1 - c v_1(\Delta) - (1-4c) S_S(v_1) = 1- 2c - (1-4c) \frac{7}{6}.
\]
Since $(S, c\Delta)$ is K-polystable, we have $\beta_{(S,c\Delta)}(v_1)\geq 0$ which implies that $c\geq \frac{1}{16}$. Thus we have $c=c_4=\frac{1}{16}$ and $\beta_{(S,c\Delta)}(v_1)= 0$. By K-polystability, we have $(S,c\Delta)\cong (S, c_4\Delta')$. If $(a_2,a_3)\neq (0,0)$, then $\mathbf{a} = (a_1,a_2,a_3)\in \bC\times (\bC^2\setminus\{(0,0)\})$ and $(S,\Delta) \cong (S,\Delta')\cong (S, \Delta_{4,\mathbf{a}})$. If $a_2=a_3=0$ and $a_1\neq 0$, then clearly $\Delta'$ admits a special degeneration to $\{y_1^2 y_1'^6 = 0\} =\Delta_{4,\infty}$ under $\sigma^{-1}$, which contradicts K-polystability. If $a_2=a_3= a_1=0$, then $\Delta' = \{y_2y_1'^6=0\}$ is toric. Then we have
\[
\beta_{(S, c \Delta')}(B) = 1 - \frac{3}{16}- (1-\frac{4}{16})\frac{7}{6}= -\frac{1}{16}<0.
\]
This implies that $(S, c\Delta')$ is K-unstable, which contradicts our assumption that $(S, c\Delta')\cong (S, c\Delta)$ is K-polystable.

If $a_0=a_0'=0$, then $a_1\neq 0$. Hence $\Delta$ admits a special degeneration to $\{y_1^2 y_1'^6 = 0\}= \Delta_{4,\infty}$ under the $\bG_m$-action $\sigma$. By similar argument as in case (a), we conclude that $(S, \Delta)\cong (S, \Delta_{4,\infty})$ and $c=c_4$.

(d) If no monomial in $f$ is divisible by $y_1'^6$, then we will show that $(S, c\Delta)$ is K-unstable for every $c\in (0, \frac{1}{16}]$.
Recall that the affine chart $U'= \{y_1'=1\}$ is isomorphic to $\{y_1^3 = y_2 y_2'\}\subset\bA^3$. Let $v_2$ be the monomial valuation on $U$ of weight $(2, 3, 3)$ in $(y_1,y_2,y_2')$ centred at $[0,1,0,0]$. Then $v_2$ is a quasi-monomial combination of $\ord_B$ and $\ord_{B'}$ of weight $(1,1)$. Thus we have
\[
\beta_S(v_2) =  \beta_S(B) + \beta_S(B') = 2\beta_S(B) = -\frac{1}{3}.
\]
Moreover, $A_S(v_2) = A_S(B) + A_S(B') = 2$ and hence $S_S(v_2) = A_S(v_2) - \beta_S(v_2) = \frac{7}{3}$.
Since  no monomial in $f$ is divisible by $y_1'^6$, we have
\[
v_2(\Delta) \geq \min \{v_2(y_1^3), v_2(y_1 y_2), v_2(y_1 y_2'), v_2(y_2^2), v_2(y_2'^2)\} = 5.
\]
Thus
\[
\beta_{(S, c\Delta)}(v_2) = 2 - cv_2(\Delta) - (1-4c)\frac{7}{3} \leq 2-5c - (1-4c) \frac{7}{3} = \frac{13c-1}{3}<0,
\]
as $c\leq \frac{1}{16}<\frac{1}{13}$. This implies that $(S,c\Delta)$ is K-unstable for every $c\in (0, \frac{1}{16}]$.
\end{proof}

\begin{proof}[Proof of Theorem \ref{thm:walls-29}]
This follows directly from Propositions \ref{prop:29-1} and \ref{prop:29-2}.
\end{proof}

\subsection{\textnumero 33} Here $S = \{x_0 x_1 -x_2 x_3 = x_1x_2 + x_2x_4 + x_3x_4=0\}\subset \bP^4$.

\begin{lem}
    Let $S$ be the surface from \textnumero 33. Then $S$ is isomorphic to the ordinary blow-up of $\bP(1,2,3)$ at two smooth points $p_1 =[1,0,0]$ and $p_2 = [1,1,0]$.
\end{lem}

\begin{proof}
    It follows from \cite[Appendix A]{CP20} or from \cite{CorayTsfasman} that the surface $S$ is a blow up of a del Pezzo surface of degree $6$ with the same singularities as $S$ at two smooth points. It is well-known that such a sextic del Pezzo surface is isomorphic to $\mathbb{P}(1,2,3)$ and its automorphism group is the Borel subgroup of the group $\mathrm{PGL}_3(\mathbb{C})$, see \cite[Big Table]{CP20}. Since the surface $S$ is also unique \cite[Big Table]{CP20}, we see that $S$ can be obtained from $\mathbb{P}(1,2,3)$ by blowing up the points $[1:0:0]$ and $[1:1:0]$ --- up to the action of $\mathrm{Aut}(\mathbb{P}(1,2,3))$ this choice of blow up points is unique.
\end{proof}

We denote by $\pi: S\to \bP(1,2,3)$ the blow-up map. Let $\sigma$ be the $\bG_m$-action on $S$ that is a lifting of the following $\bG_m$-action on $\bP(1,2,3)_{[x,y,z]}$:
\[
t\cdot [x, y, z] = [tx, t^2y, z].
\]
Let $E_1$ and $E_2$ be $\pi$-exceptional curves over $p_1$ and $p_2$, respectively. Let $C_1:=\pi_*^{-1}(y=0)$ and $C_2:=\pi_*^{-1}(y=x^2)$. Let $F:=\pi^* (x=0)$ and $B:=\pi_*^{-1}(z=0)$. Let $Q_a:=\pi^*(y=ax^2)$ for $a\in \bC\setminus \{0,1\}$.  By analysing the $\sigma$-action on $S$,  it is clear that $B$ is the only horizontal divisor on $S$, while $C_1$, $C_2$, $E_1$, $E_2$, $F$, and $Q_a$ give all vertical divisors on $S$.

The intersection numbers of these curves are summarized in Table \ref{table:No33}.
\begin{table}[ht!]
\caption{\label{table:No33} Intersection numbers of the surface \textnumero 33}
\begin{center}
\renewcommand\arraystretch{1.2}
\begin{tabular}{|c||c|c|c|c|c|}
\hline
$\bullet$  & $C_1$ & $C_2$ & $E_1$ & $E_2$ & $B$ \\
\hline\hline
$C_1$        & $-\frac{1}{3}$ & $\frac{2}{3}$ & $1$ & $0$ & $0$\\
\hline
$C_2$  & $\frac{2}{3}$ & $-\frac{1}{3}$ & $0$ & $1$ & $0$\\
\hline
$E_1$ & $1$ & $0$ & $-1$ & $0$ & $1$\\
\hline
$E_2$ & $0$ & $1$ & $0$ & $-1$ & $1$\\
\hline
$B$ & $0$ & $0$ & $1$ & $1$ & $-\frac{1}{2}$\\
\hline
\end{tabular}
\end{center}

\end{table}

\begin{thm}\label{thm:walls-33}
There is precisely one wall in $(0,\frac{1}{16}]$ involving $S$ from \textnumero 33 such that $(S,c\Delta)$ is a new K-polystable pair: $c=\frac{1}{16}$ and $\Delta=6B + 2E_1 +2E_2 +\pi^* \Gamma$ where $\Gamma =\{y^3 + b_1 x^2 y^2 + b_2 x^4 y + b_3 x^6 =0\}$ for $(b_1,b_2,b_3)\in \bC^3$ in the projective coordinates $[x,y,z]$ of $\bP(1,2,3)$.
\end{thm}

\begin{lem}\label{lem:NE33}
The Mori cone $\overline{NE}(S)$ is generated by $[C_1]$, $[C_2]$, $[E_1]$, $[E_2]$, and $[B]$.
\end{lem}

\begin{proof}
    From the $\bG_m$-action $\sigma$ on $S$, we see that the list of $\sigma$-invariant divisors are: $C_1$, $C_2$, $E_1$, $E_2$, $F$, $Q_a$, and $B$. Thus $\overline{NE}(S)$ is generated by these curves.  From the intersection table above we see that $C_1$, $C_2$, $E_1$, $E_2$, and $B$ are  all extremal. Besides, $F\sim \pi^*\cO(1)$ and $Q_a\sim \pi^*\cO(2)$ are both nef and big.
    Thus the proof is finished.
\end{proof}

\begin{prop}\label{prop:S-inv-33}
    Under the above notation, we have
    \[
    S_S(C_1) = S_S(C_2) = \frac{7}{8},  \quad S_S(E_1) = S_S(E_2) = \frac{7}{24}, \quad S_S(B)=\frac{5}{6},\quad S_S(F) = \frac{7}{6}, \quad S_S(Q_a) = \frac{7}{12}.
    \]
\end{prop}

\begin{proof}
 First of all, consider the involution $\tau:S\to S$ induced by $[x,y,z]\mapsto [x, x^2-y, z]$ on $\bP(1,2,3)$. It is clear that $\tau(C_1) = C_2$ and $\tau(E_1) = E_2$. Thus we have equalities $S_S(C_1) = S_S(C_2)$ and $S_S(E_1) = S_S(E_2)$. Besides, since $Q_a\sim 2F$, we have $S_S(Q_a) = \frac{1}{2}S_S(F)$.

Next, we compute $S_S(C_1)$. It is clear that $-K_S = \pi^*\cO(6) - E_1-E_2$.
It follows from Table \ref{table:No33} and Lemma \ref{lem:NE33} that $-K_S - tC_1$ is nef if  $0\leq t\leq 1$, and not big if $t\geq 2$. For $1\leq t\leq 2$, the Zariski decomposition $-K_S - tC_1=P(t) +N(t)$ gives the nef part
\[
 P(t) = \begin{cases}
     \pi^*\cO(6-2t) - E_2 & \textrm{ if }1\leq t\leq \frac{3}{2};\\
     (4-2t) (\pi^*\cO(3) - E_2)  & \textrm{ if }\frac{3}{2}\leq t\leq 2.
 \end{cases}
\]
Thus computation shows
\[
\vol(-K_S - tC_1) = \begin{cases}
4-2t - \frac{t^2}{3} & \textrm{ if }0\leq t\leq 1;\\
5-4t +\frac{2t^2}{3}& \textrm{ if }1\leq t\leq \frac{3}{2};\\
2(2-t)^2 & \textrm{ if }\frac{3}{2}\leq t\leq 2.
\end{cases}
\]
Thus
\[
S_S(C_1) = \frac{1}{4}\left( \int_0^1 (4-2t - \frac{t^2}{3})dt + \int_1^{\frac{3}{2}} (5-4t +\frac{2t^2}{3})dt + \int_{\frac{3}{2}}^2 2(2-t)^2 dt\right) = \frac{7}{8}.
\]

Next, we compute $S_S(E_1)$. It follows from Table \ref{table:No33} and Lemma \ref{lem:NE33} that $-K_S - tE_1$ is nef if  $0\leq t\leq 1$, and not big if $t\geq \frac{3}{2}$. For $1\leq t\leq \frac{3}{2}$, the Zariski decomposition $-K_S - tE_1=P(t) +N(t)$ gives the nef part
\[
P(t) = (3-2t)(\pi^*\cO(6)-2E_1 - E_2).
\]
Thus computation shows
\[
\vol(-K_S - tE_1) = \begin{cases}
4-2t - t^2 & \textrm{ if }0\leq t\leq 1;\\
(3-2t)^2 & \textrm{ if }1\leq t\leq \frac{3}{2}.
\end{cases}
\]
Thus
\[
S_S(E_1) = \frac{1}{4}\left( \int_0^1 (4-2t - t^2)dt + \int_1^{\frac{3}{2}}  (3-2t)^2 dt\right) = \frac{17}{24}.
\]

Next, we compute $S_S(B)$. It follows from Table \ref{table:No33} and Lemma \ref{lem:NE33} that $-K_S - tB$ is nef if  $0\leq t\leq 1$, and not big if $t\geq 2$. For $1\leq t\leq 2$, the Zariski decomposition $-K_S - tB=P(t) +N(t)$ gives the nef part
\[
P(t) = \pi^*\cO(6-3t).
\]
Thus computation shows
\[
\vol(-K_S - tB) = \begin{cases}
4-2t - \frac{t^2}{2} & \textrm{ if }0\leq t\leq 1;\\
\frac{3}{2}(2-t)^2 & \textrm{ if }1\leq t\leq 2.
\end{cases}
\]
Thus
\[
S_S(B) = \frac{1}{4}\left( \int_0^1 (4-2t - \frac{t^2}{2})dt + \int_1^{2}  \frac{3}{2}(2-t)^2 dt\right) = \frac{5}{6}.
\]

Finally, we compute $S_S(F)$. It follows from Table \ref{table:No33} and Lemma \ref{lem:NE33} that $-K_S - tF$ is nef if  $0\leq t\leq 2$, and not big if $t\geq 3$. For $2\leq t\leq 3$, the Zariski decomposition $-K_S - tF=P(t) +N(t)$ gives the nef part
\[
P(t) = (3-t)(\pi^*\cO(4)-E_1-E_2).
\]
Thus computation shows
\[
\vol(-K_S - tF) = \begin{cases}
4-2t + \frac{t^2}{6} & \textrm{ if }0\leq t\leq 2;\\
\frac{2}{3}(3-t)^2 & \textrm{ if }2\leq t\leq 3.
\end{cases}
\]
Thus
\[
S_S(F) = \frac{1}{4}\left( \int_0^2 (4-2t + \frac{t^2}{6})dt + \int_2^{3}  \frac{2}{3}(3-t)^2 dt\right) = \frac{7}{6}.
\]
\end{proof}

\begin{prop}\label{prop:ord_B-33}
    Let $S$ be the surface \textnumero 33 from Table \ref{table:dP4-Gm}. Suppose $\Delta \in |-4K_S|$ and $c\in (0,\frac{1}{16}]$ satisfy that $(S, c\Delta)$ is K-polystable but not K-stable. Then $c = \frac{1}{16}$ and $\ord_B(\Delta) =6$.
\end{prop}

\begin{proof}
    Assume that $(S, c\Delta)$ is K-polystable but not K-stable. Since $S$ admits a $\bG_m$-action $\sigma$ but is not toric, we may assume that $\Delta$ is $\bG_m$-invariant.
    It is clear that $-K_S - 2B \sim E_i + E_j$ is at the boundary of $\overline{NE}(S)$. Thus $\ord_B(\Delta) \leq 8$.

    First, we show that $\ord_B(\Delta)\neq 8$. Suppose not, then $-4K_S - 8 B \sim 4(E_1+ E_2)$ is exceptional over $\bP(1,2,3)$, which implies that $\Delta = 8B + 4E_1 + 4E_2$. By K-polystability and Proposition \ref{prop:S-inv-33}, we have
    \[
     0 = \beta_{(S,c\Delta)}(B)= 1 - 8c - (1-4c)\frac{5}{6}=0,
    \]
    which is equivalent to $c=\frac{1}{28}$. On the other hand,
    \[
    0< \beta_{(S,c\Delta)}(F) = 1 - (1-4c)\frac{7}{6}=0,
    \]
    which is a contradiction.

    Next, we show that $\ord_B(\Delta)\neq 7$. Suppose not, then $((-4K_S - 7B) \cdot E_i) = -3$ which implies that $3 (E_1 + E_2)$ is in the base component of  $|-4K_S - 7B|$. Thus we may write $\Delta =7 B +3E_1 + 3E_2 +\pi^*\Gamma_1$, where $\pi^*\Gamma_1 \sim -4K_S-7B- 3E_1 - 3E_2 \sim \pi^* \cO(3)$.
    Since $\Gamma_1$ is $\bG_m$-invariant and does not contain $\pi_* B = (z=0)$, it must contain $\pi_* F = (x=0)$ which implies that $\ord_F(\pi^*\Gamma_1)\geq 1$. By K-polystability and Proposition \ref{prop:S-inv-33}, we have
    \[
     0 = \beta_{(S,c\Delta)}(B)= 1 - 7c - (1-4c)\frac{5}{6}=0,
    \]
    which is equivalent to $c=\frac{1}{22}$. On the other hand,
    \[
    0< \beta_{(S,c\Delta)}(F) \leq  1 -c  - (1-4c)\frac{7}{6}=0,
    \]
    which is a contradiction.

    Next, we show that if $\ord_B(\Delta)=6$, then $c=\frac{1}{16}$. By K-polystability and Proposition \ref{prop:S-inv-33}, we have
    \[
     0 = \beta_{(S,c\Delta)}(B)= 1 - 6c - (1-4c)\frac{5}{6}=0,
    \]
    which is equivalent to $c=\frac{1}{16}$.

    Finally, we show that $\ord_B(\Delta)>5$. Suppose not, then we have
    \[
    0 = \beta_{(S,c\Delta)}(B) \geq 1-5c -   (1-4c)\frac{5}{6},
    \]
    which implies that $c \geq \frac{1}{10}$, a contradiction.
\end{proof}

\begin{proof}[Proof of Theorem \ref{thm:walls-33}]
Assume that $(S, c\Delta)$ is K-polystable but not K-stable where $S$ is the surface \textnumero 33 from Table \ref{table:dP4-Gm}, $\Delta\in |-4K_S|$ and $c\in (0,\frac{1}{16}]$.
By Proposition \ref{prop:ord_B-33}, we must have $c=\frac{1}{16}$ and $\ord_B(\Delta) = 6$.
Since $((-4K_S - 6B) \cdot E_i) = -2$, we know that $2 (E_1 + E_2)$ is in the base component of  $|-4K_S - 6B|$. Thus we may write $\Delta =6 B +2E_1 + 2E_2 +\pi^*\Gamma$, where $\pi^*\Gamma \sim -4K_S-6B- 2E_1 -2E_2 \sim \pi^* \cO(6)$ is $\sigma$-invariant and does not contain $B$ in its support.

Next, we classify all possible $\Gamma$ such that $(S, \frac{1}{16}\Delta)$ is K-polystable and not K-stable. Since this pair is invariant under the $\bG_m$-action $\sigma$ but not toric, by \cite[Theorem 3.2]{Liu23} (cf.\ \cite[Theorem 1.3.9]{calabi_problem_3folds} and \cite{IS17}) the pair being K-polystable is equivalent to $\beta_{(S,c\Delta)}(D) >0$  for every vertical divisor $D$ on $S$. Here we use the fact that $\beta_{(S,c\Delta)}(B)=0$ by the proof of Proposition \ref{prop:ord_B-33} and $B$ is horizontal.
A simple analysis of the $\sigma$-action on $S$ shows that $D$ is one of the following:

\begin{enumerate}[(i)]
    \item $F$;
    \item $E_1$ or $E_2$;
    \item $C_1$ or $C_2$;
    \item $Q_a$ for $a\in \bC\setminus\{0,1\}$.
\end{enumerate}

Next, we split into these four cases.

(i) By computation,
\[
    \beta_{(S,c\Delta)}(F) =  1 - \frac{1}{16}\ord_F(\pi^*\Gamma)  - (1-\frac{4}{16})\frac{7}{6}.
\]
Thus $\beta_{(S,c\Delta)}(F)>0$ is equivalent  $\ord_F(\pi^*\Gamma)<2$. Therefore, we have $\ord_F(\pi^*\Gamma) = 0$ as it is always even.

(ii) By computation,
\[
\beta_{(S,c\Delta)}(E_i) =  1 - \frac{1}{16}(2+\ord_{E_i}(\pi^*\Gamma))  - (1-\frac{4}{16})\frac{7}{24}.
\]
Thus $\beta_{(S,c\Delta)}(E_i)>0$ is equivalent to $\ord_{E_i}(\pi^* \Gamma)<\frac{21}{2}$. This is always true as  $\ord_{E_i}(\pi^*\Gamma)\leq 3$.

(iii) By computation,
\[
\beta_{(S,c\Delta)}(C_i) =  1 - \frac{1}{16}\ord_{C_i}(\pi^* \Gamma)  - (1-\frac{4}{16})\frac{7}{8}.
\]
Thus $\beta_{(S,c\Delta)}(C_i)>0$ is equivalent to $\ord_{C_i}(\pi^* \Gamma)<\frac{11}{2}$. This is always true as $\ord_{C_i}(\pi^*\Gamma)\leq 3$.

(iv) By computation,
\[
\beta_{(S,c\Delta)}(Q_a) =  1 - \frac{1}{16}\ord_{Q_a}(\pi^*\Gamma)  - (1-\frac{4}{16})\frac{7}{12}.
\]
Thus $\beta_{(S,c\Delta)}(Q_a)>0$ is equivalent to $\ord_{Q_a}(\pi^* \Gamma)<\frac{11}{2}$. This is always true as $\ord_{Q_a}(\pi^*\Gamma)\leq 3$.

To summarise, $(S, \frac{1}{16}\Delta)$ is K-polystable and not K-stable if and only if $\Gamma$ is $\bG_m$-invariant and does not contain $\pi_*B = \{z=0\}$ or $\pi_*F=\{x=0\}$ in its support. This is the same as saying that $\Gamma =\{y^3 + b_1 x^2 y^2 + b_2 x^4 y + b_3 x^6 =0\}$ for $(b_1,b_2,b_3)\in \bC^3$.
\end{proof}

\begin{proof}[Proof of Theorem \ref{thm:full-wall-crossing}]
 Let $(S, c\Delta)$ be a new K-polystable pair on a wall $c\in (0,\frac{1}{16}]$. In particular, $(S,c\Delta)$ is K-polystable with an effective $\bG_m$-action, and $S$ is K-unstable.  By Proposition \ref{prop:dP4}, we know that $S$ has at worst $A_3$ singularities.

 If $S$ has an $A_3$-singularity, the classification of walls and new K-polystable pairs  is completed in Proposition \ref{prop:dP4}, where we must have $c=\frac{1}{16}$ and
$ (S,\Delta) \cong \{(y_3^2 = y_2y_4\}\subset \bP(1,2,3,4), \{y_4^4=0\})$.
 In particular, $S$ is isomorphic to the surface \textnumero 25.

 If $S$ has no $A_3$-singularities, then Proposition \ref{prop:dP4} implies that $S$ is a K-unstable del Pezzo surface of degree $4$ with at worst $A_2$-singularities admitting an effective $\bG_m$-action. Thus by Table \ref{table:dP4-Gm} we know that $S$ is isomorphic to either \textnumero 29 or \textnumero 33. The classifications of walls and new K-polystable pairs for these two surfaces are completed in Theorems \ref{thm:walls-29} and \ref{thm:walls-33}.

 It remains to prove the statement on the initial wall $c_0=0$. Let $\cM_\epsilon$ and $\cM_0$ be the corresponding K-moduli stacks whose good moduli spaces are $\fM_\epsilon$ and $\fM_0$ respectively. Since K-semistability for coefficients is a closed condition, we know that every $[(S, \Delta)] \in \cM_\epsilon$ satisfies that $S$ is K-semistable. In particular, the forgetful map $[(S,\Delta)]\mapsto [S]$ gives a morphism between Artin stacks $\cM_\epsilon \to \cM_0$. Since good moduli spaces are initial among maps to algebraic spaces \cite[Theorem 6.6]{alper}, the forgetful map induces a morphism $\phi_0:\fM_\epsilon\to \fM_0$ between K-moduli spaces.

 Finally we analyse the map $\phi_0$. Let $[S]\in \fM_0$ be a general point, then we may assume that $S$ is a general K-stable smooth del Pezzo surface of degree $4$. Moreover, we may assume that $\Aut(S)\cong \bmu_2^4$ by \cite[Section~6.4]{DolgachevIskovskikh}. By \cite{Zhou21}, we know that there exists $\epsilon_0>0$ such that $(S, \epsilon \Delta)$ is K-stable for every $\Delta\in |-4K_S|$ and every $\epsilon\in (0, \epsilon_0)$. This implies that the finite quotient $\bP(H^0(S, -4K_S))/\Aut(S)$ admits a closed immersion into the fibre $\phi_0^{-1}([S])$. On the other hand, if $[(S', \Delta')]\in \phi_0^{-1}([S])$, then $S'$ and $S$ are $S$-equivalent K-semistable Fano varieties. Since $S$ is K-stable, by \cite{lwx, blum_xu_uniqueness} we know that $S'\cong S$ which implies that $[(S',\Delta')]$ is contained in the image of $\bP(H^0(S, -4K_S))/\Aut(S)$. Thus we have $\bP(H^0(S, -4K_S))/\Aut(S)\cong \phi_0^{-1}[S]$. By a Riemann-Roch computation, we have $h^0(S, -4K_S) = 41$ which implies that $\bP(H^0(S, -4K_S))/\Aut(S)\cong \bP^{40}/\bmu_2^4$. Since both $\fM_c$ and $\fM_0$ are normal by \cite[Theorem 2.21]{K-moduli_quadric_K3}, the statement follows by Zariski's main theorem.
\end{proof}

\begin{remark}
In \S\ref{sec:toric-examples}, we found three K-polystable toric Fano \threefolds{}  of the form $X_{2,2,4}\subset \bP(1^5, 2^2)$ defined by $y_0y_1 = f(x_0, \cdots, x_4)$, $g(x_0,\cdots, x_4)=0$ and $h(x_0, \cdots, x_4)=0$, where $f$, $g$, and $h$ are of the following:
\begin{enumerate}
    \item $f=x_0^4$, $g=x_2^2 - x_0 x_3$, and $h = x_3^2 - x_1x_4$;
    \item $f = x_0^3 x_3$, $g = x_2^2 - x_0 x_3$, and $h = x_2 x_3 - x_1 x_4$;
    \item $f = x_0^2 x_4^2$, $g=x_2^2 - x_0 x_4$, and $h = x_1 x_3 - x_0 x_4$.
\end{enumerate}
By Proposition \ref{prop:surface-3fold} and Corollary \ref{cor:moduli}, each of the above toric Fano \threefolds{} corresponds to a K-polystable pair $(S, \frac{1}{16}\Delta)$ in $\fM_{\frac{1}{16}}$ where $S=\{g=h=0\}\subset \bP^4$ and $\Delta = \{f=0\}$. It turns out that these three surface pairs are precisely the toric ones in $\fM_{\frac{1}{16}}$:
\begin{enumerate}
    \item corresponds to the surface \textnumero 25 and $\Delta = \{y_4^4=0\}$ via  $[x_0, x_1,x_2,x_3,x_4] = [y_4, y_2^2, y_1y_3, y_1^2 y_2, y_1^4]$;
    \item corresponds to the surface \textnumero 29 and $\Delta = \Delta_{4,\infty}$ via $[x_0, x_1,x_2,x_3,x_4] = [y_1'^2, y_2, y_1y_1', y_1^2, y_2']$;
    \item corresponds to the surface \textnumero 30 and $\Delta = \{y_1^4 y_1'^4=0\}$ via $[x_0, x_1,x_2,x_3,x_4] = [y_1^2, y_2, y_1y_1', y_2', y_1'^2]$.
\end{enumerate}
Here K-polystability of $(S,\frac{1}{2}\Delta)$ from (1) or (2) follows from Theorem \ref{thm:full-wall-crossing}, while (3) follows by interpolation \cite[Proposition 2.13]{wall_crossing_K_moduli} as  \textnumero 30 is K-polystable, and $\frac{1}{4}\Delta$ is the reduced toric boundary.
\end{remark}

\subsection{Further discussions}
Let $M$ be the K-moduli component of quartic \threefolds{}.
While our results provide a decisive statement about the global nature of objects in $M$, several questions remain open. For instance, it would be interesting to determine all the pairs parametrised by $\fM_{\frac{1}{16}}$ via wall crossings. %, which we think should be isomorphic to some variation of geometric invariant theory for the pair $(S,\Delta)$.
Although Theorem \ref{thm:full-wall-crossing} gives a complete description of walls and new K-polystable pairs, one would need to describe the exceptional loci for K-moduli wall crossings, which is similar to the VGIT analysis on basins and attractions.  %the explicit wall crossings from $\fM_{\epsilon}$ to $\fM_{\frac{1}{16}}$.
We leave this for future studies. %Note that, similar questions were considered in \cite{K-moduli_quadric_K3}.

It would also be handy to have an explicit description of GIT-stable quartic \threefolds{}. Indeed, we think that the locus in $M$ parametrising pure $(2,2,4)$-complete intersections $X_{2,2,4} \subset \PP(1^5,2^2)$ should correspond to the locus of the union of two quadric hypersurfaces $\{gh = 0\}\subset \bP^4$ in the GIT moduli space of quartic \threefolds{} via wall crossing. On the one hand, $\{gh=0\}\subset \bP^4$ degenerates to $\{y_0y_1=0, ~g=h=0\} \subset \bP(1^5,2^2)$  via iterated degeneration to normal cone. On the other hand, $X_{2,2,4}=\{y_0y_1=f,~g=h=0\}\subset \bP(1^5,2^2)$ also degenerates to $\{y_0y_1=0, g=h=0\} \subset \bP(1^5,2^2)$ via a suitable one parameter subgroup. Note that $\dim M = 45$, and the dimension of the locus parametrising pure $(2,2,4)$-complete intersections in $\PP(1^5,2^2)$  is equal to $\dim \fM_{\frac{1}{16}}$ which is $42$.

In a different direction, the following example from \cite[Remark~6.13]{adl_quartic_K3} gives another closed locus in $M$.

\begin{example}
In \cite{adl_quartic_K3}, the K-moduli compactification $\overline{\fM}_{\frac{3}{4}}^{\rm K}$ of $(\bP^3, \frac{3}{4}D)$ where $S$ is a quartic surface was studied. If $(X, \frac{3}{4}D)$ is a K-polystable \threefold{} pair in $\overline{\fM}_{\frac{3}{4}}^{\rm K}$, then we can take the quadruple cyclic cover $Y\to X$ branched along $D$.
By  \cite{LZ22, Zhuang} we know that $Y$ is K-polystable.
Since a quadruple cyclic cover of $\bP^3$ branched along a quartic surface $\{f(x_0, x_1,x_2,x_3)=0\}$ is precisely the quartic \threefold{} $\{x_4^4 = f(x_0,x_1,x_2,x_3)\}\subset\bP^4$, we know that $Y$ is a K-polystable limit of a family of smooth quartic \threefolds{}.
This cyclic cover construction can be made fibrewise (see \cite[\S6.3]{adl_quartic_K3} for the case of quartic double solids), which yields a finite morphism $\iota: \overline{\fM}_{\frac{3}{4}}^{\rm K}\to M $. It is clear that the image of $\iota$ has dimension $19$ and is a K-moduli compactification of smooth quartic \threefolds{} of the form $\{x_4^4 = f(x_0,x_1,x_2,x_3)\}\subset \bP^4$. By \cite[Theorem 5.16]{adl_quartic_K3}, we know that there are four of the Fano \threefold{} $X$ appearing in $\overline{\fM}_{\frac{3}{4}}^{\rm K}$: $\bP^3$, $X_h$, $\bP(1,1,2,4)$, or $X_u$. Here $X_h = \{x_0 x_1 =x_2 x_3\}\subset \bP(1^4, 2)$ which is isomorphic to the projective cone over $\bP^1\times \bP^1$ with polarisation $\cO(2,2)$; $X_u$ is a $2$-step birational modification of a $\bP(1,2,3)$-bundle over $\bP^1$ constructed in \cite[\S4.2]{adl_quartic_K3}. One can check that the quadruple cyclic cover of $X_h$ is a $(2,4)$-complete intersection of $\bP(1^5, 2)$, while the quadruple cyclic cover of $\bP(1,1,2,4)$ is a weighted hypersurface of degree $8$ in $\bP(1,1,2,2,4)$. However, we expect that the quadruple cyclic cover of $X_u$ will not be a weighted complete intersection. Such quadruple cyclic covers of $X_u$ form an $18$-dimensional closed locus in the K-moduli space $M$.
\end{example}

%In addition, we would need to describe the exceptional loci for K-moduli wall crossings, which is similar to the VGIT analysis on basins and attractions.\footnote{\YL{It could be fairly complicated to describe all exceptional loci.}}

%\YL{add ADL21, deg 4 cyclic cover thing and $X_u$}

\begin{comment}
	\xymatrix{
		&&&&\widetilde{X}\ar@{->}[dd]^{\eta}\ar@{->}[drr]^{\eta_1}\ar@{->}[dll]_{\eta_0}&&&&\\%
		&&U\ar@{->}[dd]_{\phi}\ar@{->}[dll]_{\psi}&&&&U\ar@{->}[dd]^{\phi}\ar@{->}[drr]^{\psi}&&\\%
		V\ar@{-->}[drr]_{\tau}&&&& S\ar@{=}[rrd]&&&&V\ar@{-->}[dll]^{\tau}\\
		&&S\ar@{=}[rru]&&&&S&&}
\end{comment}

\subsection*{Acknowledgements}
We thank Thomas Hall and Anne-Sophie Kaloghiros for fruitful conversations.

This project was in-part carried out whilst HA visited the Institut de Math\'{e}matiques de Toulouse, supported by LabEx CIMI; IC visited the Institut des Hautes \'{E}tudes Scientifiques; YL visited the University of Edinburgh; and AP visited Osaka University and Kumamoto University. The authors would like to thank the members of these institutions for creating stimulating mathematical environments and excellent scientific activities. HA is supported by EPSRC Grant~EP/V048619. IC is supported by EPSRC Grant~EP/V054597. AK~is supported by EPSRC Fellowship~EP/N022513/1. YL is partially supported by NSF grant DMS-2148266 (formerly DMS-2001317) and an Alfred P. Sloan research fellowship. AP is supported by INdAM GNSAGA ``Gruppo Nazionale per le Strutture Algebriche, Geometriche e le loro Applicazioni'' and PRIN2020 2020KKWT53 ``Curves, Ricci flat Varieties and their Interactions''.

%-------------------------------------------------------------------------------
\bibliographystyle{plain}
\bibliography{Biblio_quartic_3folds}
%-------------------------------------------------------------------------------
\end{document}